\documentclass[12pt,a4paper,reqno]{amsart} 
 \usepackage{amsmath, amssymb, amsthm}
\usepackage[english]{babel} 
\usepackage{fullpage}
\usepackage{hyperref}  
\usepackage{enumerate}
\usepackage[all]{xy}
\usepackage{color}
\usepackage{wasysym}

\title[Title] 
{On the geometry of \\~\\ quantum spheres and hyperboloids} 

\date{February 2024}

\author{Giovanni Landi, Chiara Pagani}

\address[]{\textit{Giovanni Landi}  \newline \indent 
Matematica, Universit\`a di Trieste, \newline \indent
Via A. Valerio, 12/1, 34127  Trieste, Italy \newline \indent
and INFN, Trieste, Italy}
\email{landi@units.it}
\address[]{\textit{Chiara Pagani} 
\newline \indent    Alma Mater Studiorum Universit\`a di Bologna,
\newline \indent Dipartimento di
Matematica,  Piazza di Porta S. Donato, 5, 40126 Bologna, Italy.}
\email{c.pagani@unibo.it ,  cpagani@units.it}

\numberwithin{equation}{section}

\theoremstyle{plain}
\newtheorem{thm}{Theorem}[section]
\newtheorem{lem}[thm]{Lemma}
\newtheorem{prop}[thm]{Proposition}

\theoremstyle{definition}

\newtheorem{rem}[thm]{Remark}

\newcommand{\nn}{\nonumber}
\newcommand{\ot}{\otimes}
\newcommand{\beq}{\begin{equation}}
\newcommand{\eeq}{\end{equation}}
 \newcommand{\id}{\mathrm{id}}
\newcommand{\IC}{\mathbb{C}}
\newcommand{\IN}{\mathbb{N}}
\newcommand{\IR}{\mathbb{R}}
\newcommand{\IZ}{\mathbb{Z}}

\newcommand{\half}{{\frac{1}{2}}}

\newcommand{\II}{\mathrm{I}}
\newcommand{\ii}{\mathrm{i}}
\newcommand{\cA}{A}
\newcommand{\cU}{\mathcal{U}}
\newcommand{\cO}{\mathcal{O}}
\newcommand{\cC}{\mathcal{C}}

\DeclareMathOperator{\chern}{ch}
\DeclareMathOperator{\tr}{tr}
\DeclareMathOperator{\M}{Mat}

\newcommand{\s}{\mathsf{s}}
\newcommand{\slq}{\mathcal{U}_{q}(sl_2)}
\newcommand{\slqh}{\mathcal{U}_{q^{1/2}}(sl_2)}

\newcommand{\uqh}{\mathcal{U}_{q^{1/2}}}
\newcommand{\qvs}{\IC^3_q}
\newcommand{\aqvs}{\cO(\IC^3_q)}

\newcommand{\zero}[1]{{#1}_{\scriptscriptstyle{(0)}}}
\newcommand{\one}[1]{{#1}_{\scriptscriptstyle{(1)}}}
\newcommand{\mone}[1]{{#1}_{\scriptscriptstyle{(-1)}}}

\newcommand{\qvsN}{\IC^N_q}
\def\oq#1{O_q({#1})}
\def\soq#1{SO_q({#1})}
\def\aoq#1{\cO(O_q({#1}))}
\def\asoq#1{\cO(SO_q({#1}))}

\def\asphf{\cO(S^2_q)}  
\def\ahypf{\cO(H^2_q)}  
\def\sphf{S^2_q}  
\def\hypf{H^2_q}

\def\asph{\cO(S^2_{q,Gr})}           
\def\ahyp{\cO(H^2_{q,Gr})}           
\def\sph{S^2_{q,Gr}}           
\def\hyp{H^2_{q,Gr}}           

\def\bra#1{\left\langle #1\right|}
\def\ket#1{\left| #1\right\rangle}
\def\hs#1#2{\left\langle #1,#2\right\rangle}

\def\u#1#2{u_{#1 #2}}
\def\v#1#2{\tilde{u}_{#1#2}}
\def\x#1{x_{#1}}
\def\y#1{y_{#1}}
\def\X#1{X_{#1}}
\def\Y#1{Y_{#1}}

\def\hu#1#2{\widehat{u}_{#1 #2}}

\begin{document}


\begin{abstract}
We study two classes of quantum spheres and hyperboloids which are $*$-quantum spaces for the  quantum orthogonal group   $\mathcal{O}(SO_q(3))$.
We construct line bundles over the quantum homogeneous space of invariant elements 
for the quantum subgroup $SO(2)$ of $SO_q(3)$.
These are associated to the quantum principal bundle via corepresentations of $SO(2)$ and are given by finitely-generated projective modules $\mathcal{E}_n$ of rank $1$ and even degree $-2n$. The corresponding idempotents, representing classes in K-theory, are explicitly worked out. 
For $q$ real, we diagonalise the Casimir operator of the 
Hopf algebra ${\mathcal{U}_{q^{1/2}}(sl_2)}$ dual to $\mathcal{O}(SO_q(3))$.
\end{abstract}

\maketitle

\tableofcontents
\parskip = .75 ex
\allowdisplaybreaks[4]

\section{Introduction}
This paper is part of a scientific programme which deals with Laplacian operators on quantum  spaces with quantum group symmetries. Here we study two classes of quantum spheres and hyperboloids with 
symmetry from the quantum orthogonal group $\soq{3}$.  

In the approach of \cite{frt}, the quantized algebra of functions  $\asoq{N}$ on the quantum orthogonal group in any dimension is given
as the algebra generated by $N^2$ elements subject to commutation relations that 
depend on the entries of a matrix $R$ which is a solution of the quantum Yang--Baxter equation.
The matrix $R$ decomposes in terms of projections and this allows one to introduce 
quantum spaces carrying natural coactions of the quantum group $\asoq{N}$ (see \S \ref{sec:qvs}). 

When restricting to $\asoq{3}$, a first class of quantum spheres and hyperboloids is obtained as real forms of the quantum vector  spaces of  $\asoq{3}$ associated with  the $q$-symmetrizer projection $P_-$ in the decomposition of the $R$-matrix alluded to before.  The nature of the quantum space is determined by the $*$-structure: for $q \in \IR$ one gets a sphere  --- the equatorial Podle\'s sphere, while for $|q|=1$ an hyperboloid.  

A second class, described in \S \ref{sec:qhs}, is given by quantum homogeneous $\asoq{3}$--spaces arising from the coaction of the quantum subgroup $SO(2)$ of $\soq{3}$ on the latter.
Again, the $*$-structure discriminates  between a quantum $2$-sphere --- now the standard Podle\'s sphere, and an hyperboloid. In both cases, the quantum homogenous space, given as the subalgebra $B$ of coinvariants of $\asoq{3}$ for the right coaction of $SO(2)$, is explicitly determined.
This also makes use of the identification of 
$\mathcal{O}(SL_{\s}(2))$, for $\s=q^\half$, as the `double covering' of $\asoq{3}$, that is of
the existence of a
 Hopf algebra isomorphism between the coordinate algebra $\mathcal{O}(SO_{q}(3))$ and the subalgebra of $\mathcal{O}(SL_{\s}(2))$ made of invariant elements for the action of the group $\IZ_2$ (see \S \ref{sec:dc}).

The algebra extension $B \subset \asoq{3}$ is shown to be an $SO(2)$ quantum principal bundle (a Hopf--Galois  extension). 
This quantum principal bundle has associated (modules of sections of) line bundles coming from the corepresentations of $SO(2)$. The module are given by finitely-generated projective modules $\mathcal{E}_n$  of rank $1$ and even degree $-2n$.  The corresponding 
idempotents  $p_n \in \M_{| 2n | +1}(B)$, describing classes in the K-theory of the algebra  $B$, are explicitly worked out. 
These idempotents are different from those usually used for Podle\'s sphere, a fact that reflects in a simpler recursion formula for their trace and thus for an easier computation of their degree (Proposition \ref{prop:pn}).

 For the study of Laplacian operators on the two $*$-quantum homogeneous spaces  of $\asoq{3}$ and of `gauged' Laplacian operators on bundles over them (in the line of \cite{lrz}),  the last section of the paper is dedicated to the  study of the quantum 
Casimir element of $\slqh$, the Hopf algebra dual to the Hopf algebra $\asoq{3}$. For $q$ real, the Casimir operator, which acts on the left on $B$ and on lines bundles over the latter, is diagonalised via the commuting right action of 
$\slqh$ (Theorem \ref{cC-diag}).
 
\section{The quantum special orthogonal groups  $\soq{N}$} \label{qogn}

We recall 
 the construction of the coordinate algebra $\aoq{N}$ of the quantum orthogonal group $\oq{N}$; see e.g. \cite[\S 9.3]{ks}.
Let $q \in \IC $, $q \neq 0$, fixed. Let $N$ be an integer. For each index $i=1,\dots ,N$, let 
$i'=N+1-i $ and define $\rho_i=\frac{N}{2}-i  $ if $i <i'$, with $\rho_{i'}=- \rho_i$ and  $\rho_i=0$ if $i = i'$.
For all indices $i,j,m,n=1, \dots, N$ we define complex numbers 
\beq\label{Rmatrix}
R^{ij}_{mn}=q^{\delta_{ij}-\delta_{ij'}} \delta_{im}\delta_{jn} + (q-q^{-1}) \theta(i-m)
(\delta_{jm}\delta_{in}-q^{-\rho_j-\rho_m}\delta_{ij'}\delta_{nm'})
\eeq
where $\theta$ is the Heaviside function, whose value is one for strictly positive argument and zero otherwise.
We then consider the free algebra $\IC \langle \u{i}{j} \rangle$ generated over $\IC$ by $N^2$ elements $\u{i}{j} $, $i,j=1, \dots, N$, modulo the two-sided ideal generated by elements
\beq\label{rtt}
R^{ji}_{kl} \u{k}{m}  \u{l}{n} - \u{i}{k}  \u{j}{l}  R^{lk}_{mn} \; , \quad  ~i,j,m,n=1,\dots,N \;.
\eeq
Explicitly, the quotient algebra, that we denote by $\mathcal{O}(R)$, is generated by elements $\u{i}{j}$ subject to relations
\begin{multline}\label{rtt-exp}
q^{\delta_{ij}-\delta_{ij'}} \u{j}{m}  \u{i}{n}  = ~  q^{\delta_{mn}-\delta_{mn'}}  \u{i}{n} \u{j}{m} 
+ \lambda \left(\theta(n-m)- \theta(j-i)\right) \u{i}{m}  \u{j}{n} 
\\  + \lambda \delta_{ij'} \sum_k \theta(j-k)  q^{-\rho_i-\rho_k}\u{k}{m}  \u{k'}{n}   
- \lambda \delta_{nm'} \sum_k \theta(k-m)  q^{-\rho_m-\rho_{k'}}\u{i}{k'}  \u{j}{k} \, , 
\end{multline}
where we set $\lambda:= q - q^{-1}$. In concise matrix notations, $\mathcal{O}(R)$ is the algebra generated by the entries of the $N \times N$ matrix $u=(\u{i}{j} )$ with relations
\beq\label{rtt-matrix}
R u_1 u_2=u_2 u_1 R \; ,
\eeq
for $R$ the  $N^2 \times N^2$ matrix of entries  $R=(R^{ij}_{mn})$ (where $i$ is a row block index, $m$ a column block index  and $j,n$ are respectively the row and column index inside each block) and   
$u_1=u \ot \II$, $u_2 = \II \ot u$ with $\II$ the unit matrix.

The algebra $\aoq{N}$ is then the quotient algebra of  $\mathcal{O}(R)$ which is obtained by imposing the generators $\u{i}{j}$ to satisfy the additional  orthogonality (metric) condition
\beq\label{metric}
uCu^tC^{-1} = \II = C^{-1}u^tC u \, , 
\eeq
with matrix 
\beq
C=(C_{k j}) \, ,  \quad C_{k j} = \delta_{k j'} q^{-\rho_k} 
\eeq
(and then $C=C^{-1}$). 
In the classical case the condition \eqref{metric} is just the metric condition defining the complex Lie group $\cO(N,\IC)$. 

 The metric condition \eqref{metric} corresponds to one single additional relation $Q_q-1=0$ (see \cite[page 319]{ks}), where $Q_q$ can equivalently be expressed in terms of any index $j$ as 
\beq\label{condq}
Q_q= \sum_k C_{j' j} C_{k k'} \u{k}{j}  \u{k'}{j'} = \sum_k C_{j' j} C_{k k'} \u{j}{k} \u{j'}{k'} \, .
\eeq

The algebra $\aoq{N}$ is a Hopf algebra with coproduct $\Delta$, counit $\varepsilon$ and  antipode $S$ given 
on generators respectively by
\beq
\Delta(\u{k}{j})= \sum_m \u{k}{m}  \ot \u{m}{j}  \, ,  \quad \varepsilon(\u{k}{j})= \delta_{ij}
\, ,  \quad
S(\u{k}{j})= q^{\rho_j-\rho_k} \u{j'}{k'}
\eeq
or in matrix notation 
$$
\Delta(u)= u \ot u 
\, ,  \quad \varepsilon(u)=\II \, ,  \quad 
S(u)= Cu^t C^{-1} \, .
$$

\subsection{Real forms}
The coordinate algebra $\aoq{N}$ of the quantum orthogonal group $\oq{N}$ admits different $*$-structures $*: \aoq{N} \to \aoq{N}$, resulting in different real forms (see \cite[\S 9.3.5]{ks}).
For the present paper we consider the following two choices:

For $q \in \IR$, define
\beq\label{*-struct1}
(\u{j}{k})^*:= S(\u{k}{j})=q^{\rho_j-\rho_k} \u{j'}{k'} .
\eeq
The resulting Hopf $*$-algebra is the coordinate algebra $\cO(O_q(N, \IR))$ of the compact quantum group $O_q(N, \IR)$, with  defining matrix $u$ which is unitary, $uu^\dag=\II=u^\dag u$, with $(u^\dag)_{kj}:=(\u{j}{k})^*= 
S(\u{k}{j})$. 

For $|q|=1$, define
\beq\label{*-struct2}
(\u{j}{k})^*:= \u{j}{k} .
\eeq
The resulting Hopf $*$-algebra is the coordinate algebra $\cO(O_q(n,n))$  
of the real  quantum group $O_q(n,n)$ for $N=2n$ even, or $\cO(O_q(n,n+1))$ of the real  quantum group $O_q(n,n+1)$ for $N=2n+1$ odd.

\subsection{Quantum spaces and exterior algebras}
We  recall from \cite{frt} (see also \cite[\S 8.4.3, \S 9.1.2]{ks}) that the matrix $R$ satisfies a cubic equation,
$$
(\widehat{R} -q \II)(\widehat{R} +q^{-1} \II)(\widehat{R} -q^{1-N} \II)=0 \, .
$$
in terms of the matrix $\widehat{R}=  (\widehat{R}^{kj}_{mn}) :=(R^{jk}_{mn})$. 
Moreover for $N>2$, and assuming 
$(1+q^2)(1+q^{-1})(1-q^{-3}) \neq 0$, a condition that implies in particular that $q - q^{-1} \neq 0$, the matrix $\widehat{R}$ is semisimple and can be decomposed as 
\beq\label{decr}
\widehat{R}=q P_+ -q^{-1} P_- +q^{1-N} P_0 \; , \quad 
\eeq
with $P_\alpha$, $\alpha = \pm, 0$ mutually orthogonal idempotents:
$P_\alpha^2=P_\alpha $, and $P_\alpha P_\beta=0$, for  $\alpha \neq \beta$.
In the decomposition \eqref{decr}, the matrix $P_-$ is 
a $q$-symmetrizer matrix on $\IC^N \times \IC^N$
and is used to define  a quantum  space $$V=\qvsN:= \IC \langle \x{m}\rangle /{\langle (P_-)^{jl}_{mn} \x{m}\x{n} \rangle} \;  , \qquad m,n,j,l=1, \dots N,
$$
while $P_+$ and $P_0$ are used to define a quantized orthogonal exterior algebra 
\beq\label{extV}
\Lambda_q(V):= \IC \langle e_{m}\rangle /{\langle (P_+)^{jl}_{mn} e_{m}e_{n},~~
(P_0)^{jl}_{mn} e_{m}e_{n} \rangle} \; .
\eeq
 Both $V$ and $\Lambda_q(V)$ carry a left coaction of $\aoq{N}$ 
 given by the algebra morphisms 
 $$
 x_j \mapsto \sum_k \u{j}{k} \ot x_k \; , \quad e_j \mapsto \sum_k \u{j}{k} \ot e_k .
 $$
In particular, the subspace of $\Lambda_q(V)$ made of degree $N$ polynomials is one-dimensional and thus there exists a unique element $D_q(u)\in \aoq{N}$ such that for each element $\xi$ in  
 $\Lambda_q(V)$ of degree $N$, the coaction is simply given by $\xi \mapsto D_q(u) \ot \xi$.
The element $D_q(u)$ is called the quantum determinant of the matrix $u$. It is shown to belong to the centre of the algebra  $\aoq{N}$ and to be group-like, that is $\Delta(D_q(u))= D_q (u)\ot D_q(u)$ and $\varepsilon(D_q(u))= 1$.

The two-sided ideal generated by $\langle D_q(u)-1 \rangle$ is a Hopf ideal of $\aoq{N}$ and the  quotient  Hopf algebra $\aoq{N}/ \langle D_q(u)-1 \rangle$ is called the coordinate algebra $\asoq{N}$ of the special orthogonal quantum group $\soq{N}$.  

\section{The quantum orthogonal group $\soq{3}$} 
We specialize the above to the case 
$N=3$. For each index $i=1,2,3$, one has $i'=3-i$ so that  $1'=3$, $2'=2$ and 
$\rho_1=\half \; , \; \rho_2=0 \; , \; \rho_3=-\half$. 
The matrix $R=(R^{kj}_{mn})$ is the lower-diagonal matrix  
\beq\label{Rmat3}
{R}= \left(
\begin{array}{ccc|ccc|ccc}
q &  &  &&&&&&
\\
0 & 1 &  &&&&&&
\\
0 &0 & q^{-1} &  &  &&&&
\\
\hline
0 & \lambda & 0 &1&&&&&
\\
0 & 0 & -q^{\half} \lambda &0&1&&&&
\\
0 &0 & 0 & 0 & 0 &1&&&
\\
\hline
0 & 0 & 0 &0&-q^{\half} \lambda&0&q^{-1}&&
\\
0 & 0 & 0 &0&0& \lambda&0&1&
\\
0 &0 & 0 & 0 & 0 &0&0&0&q
\end{array}\right)  
\eeq
(where $\lambda= q - q^{-1}$ as before)
 with non-zero entries
$$
\begin{array}{lllll}
& R^{11}_{11}= R^{33}_{33}=q  \qquad & R^{13}_{13}= R^{31}_{31}=q^{-1}
 \quad 
 & R^{22}_{22}=1  \qquad &  R^{12}_{12}= R^{21}_{21}= R^{23}_{23}= R^{32}_{32}=1 \\
  & R^{21}_{12}= \lambda   & R^{22}_{13}= - q^{\half} \lambda 
  & R^{32}_{23}= \lambda    & R^{31}_{22}= - q^{\half} \lambda .
\end{array}
$$

According to the general theory, the Hopf algebra $\aoq{3}$ is
the  free algebra generated  by elements $\u{i}{j}$, $i,j=1,2,3$ modulo the ideal of  relations \eqref{rtt}
and \eqref{metric} (or \eqref{condq}), 
$\langle R u_1 u_2-u_2 u_1 R\, , \;
uCu^tC^{-1}-\II \, ;  \; C^{-1}u^tC u=\II \rangle$. In matrix form the antipode is
$$S(u)= Cu^t C^{-1}
=
\begin{pmatrix}
\u{3}{3}  & q^{-\half}  \u{2}{3}  & q^{-1}\u{1}{3} 
\\
q^{\half}  \u{3}{2}  & \u{2}{2}  & q^{-\half} \u{1}{2} 
\\
q~ \u{3}{1}  & q^{\half}  \u{2}{1}  & \u{1}{1}  
\end{pmatrix} \, .
$$

\subsection{The quantum determinant}\label{sec:q-det}
From the decomposition \eqref{decr} of the matrix in \eqref{Rmat3}, one gets  a quantum  space $V=\qvs$, and an exterior algebra $\Lambda_q(V)$, both carrying a right coaction of $\aoq{3}$.
We will  return to $\qvs$ in \S \ref{sec:qvs} below. Here we consider the 
exterior algebra $\Lambda_q(V)$ in \eqref{extV}, which allows one to define the  quantum determinant $D_q(u)$. 

The graded algebra $\Lambda_q(V)$ is generated in degree one by  elements $e_1,e_2,e_3$ with relations 
\beq\label{comm-e}
\begin{array}{lll}
 (e_1)^2  =0\, ,  \; &(e_3)^2=0 \, ,  \; &(e_2)^2=(q^\half - q^{-\half}) e_1 e_3~
; \nn \\
 e_3  e_2  =- q e_2 e_3\, ,  \;
& e_3  e_1=- e_1 e_3\, ,  \; & e_2  e_1=- q e_1 e_2\, ,  
\end{array}
\eeq
and coaction of $\aoq{3}$ given by $\rho: e_j \mapsto \sum \u{j}{k} \ot e_k$ on the generators and  extended  to the whole $\Lambda_q(V)$ as an algebra map. 
Out of the commutation relations \eqref{comm-e}, it follows that in degree three all elements are proportional:
$$ e_k e_m e_n= \varepsilon_{kmn} ~w \qquad \mbox{ for (say) } w:=e_1 e_2 e_3 , \quad \forall k,l,m=1,2,3.
$$ The only non zero components of the tensor $\varepsilon$ are found to be
\begin{align}\label{tensor-epsilon}
\varepsilon_{123}=1  \, ,  \;  \varepsilon_{132}=-q  \, ,  \; \varepsilon_{213}=-q  \, ,  \; 
\varepsilon_{231}=q  \, ,  \; \nn \\ \varepsilon_{312}=q  \, ,  \;  \varepsilon_{321}=-q^2  \, ,  \; 
\varepsilon_{222}=-q (q^\half - q^{-\half})  \, .
\end{align} 
Hence  there exists a unique element $D_q(u) \in \aoq{3}$ such that 
$\rho(\xi) = D_q(u) \ot \xi$ for each $\xi$ monomial in $\Lambda_q(V)$ of degree three. 
For $\xi=w=e_1 e_2 e_3$ one promptly obtains  the following explicit formula for 
the  quantum determinant 
$D_q(u)$:
\begin{align}\label{det}
D_q(u)= &~ \u{1}{1} \u{2}{2} \u{3}{3}  -q \u{1}{2} \u{2}{1} \u{3}{3}  -q\u{1}{1} \u{2}{3} \u{3}{2}  
+q \u{1}{2} \u{2}{3} \u{3}{1}  \nn \\
& \qquad + q \u{1}{3} \u{2}{1} \u{3}{2}  -q^2\u{1}{3} \u{2}{2} \u{3}{1} 
-q (q^\half-q^{-\half})\u{1}{2} \u{2}{2} \u{3}{2}  \, .
\end{align}
The  quotient  Hopf algebra $\aoq{3}/ \langle D_q(u)-1 \rangle$ is   the coordinate algebra $\asoq{3}$ of the special orthogonal quantum group $\soq{3}$.

The determinant $D_q(u)$ admits different equivalent expressions as a degree three polynomial on the generators $\u{j}{k}$ of $\aoq{3}$: for each triple of indices $a,b,c=1,2,3$ such that $\varepsilon_{abc} \neq 0$, being $\rho: e_j \mapsto \sum_k \u{j}{k} \ot e_k$, one computes
$$
\rho(e_a e_b e_c)= \sum_{m,n,p} \u{a}{m}\u{b}{n}\u{c}{p}  \ot e_m e_n e_p
=\sum_{m,n,p} \u{a}{m}\u{b}{n}\u{c}{p}  \ot \varepsilon_{mnp} w
$$
and therefore, 
\beq\label{cofactors}
D_q(u)= \sum_m \u{a}{m} \hu{m}{a} \; , \qquad \text{with} \qquad
\hu{m}{a}:= \varepsilon_{abc}^{-1} \sum_{n,p} \varepsilon_{mnp}  \u{b}{n}\u{c}{p}  \;.
\eeq
We refer to this formula $D_q(u)= \sum_m \u{a}{m} \hu{m}{a} $ as the expansion of $D_q(u)$ with respect to the $a$-row and we call the element $\hu{m}{a} $ the cofactor of $\u{m}{a}$ and 
$\mathsf{cof}(\u{}{}):={^t\hu{}{}}$ the matrix of cofactors.
Notice that each cofactor $\hu{m}{a} $ admits more than one expression, one for each possible choice of indices $b,c$ such that $\varepsilon_{abc} \neq 0$: for each $m=1,2,3$ one computes
\begin{align*}
\hu{m}{1} &= \sum_{n,p}  \varepsilon_{mnp} \u{2}{n}\u{3}{p}  =
-q^{-1} \sum_{n,p} \varepsilon_{mnp}  \u{3}{n}\u{2}{p} 
\\
\hu{m}{2} &= -q^{-1} \sum_{n,p} \varepsilon_{mnp}  \u{1}{n}\u{3}{p}  =
q^{-1} \sum_{n,p} \varepsilon_{mnp}  \u{3}{n}\u{1}{p}  =
-q^{-1} (q^\half -q^{-\half})^{-1} \sum_{n,p} \varepsilon_{mnp}  \u{2}{n}\u{2}{p} 
\\
\hu{m}{3} &= q^{-1} \sum_{n,p} \varepsilon_{mnp}  \u{1}{n}\u{2}{p}  =
-q^{-2} \sum_{n,p} \varepsilon_{mnp}  \u{2}{n}\u{1}{p} \, .
\end{align*}
We explicitly list all of them in Appendix~\ref{app:B}.

The matrix $\widehat{u}$ of cofactors can be identified with the antipode matrix. For this we need the following result for which we use explicit commutation relations of the type 
\eqref{rtt-exp} with the matrix \eqref{Rmat3} as well as the orthogonality conditions.
\begin{prop}\label{pro:cof}
Let $\widehat{u}= (\hu{j}{k})_{j,k=1,2,3}$ be the transpose of the matrix
of  cofactors, $\hu{m}{a}=\mathsf{cof}(\u{}{})_{am}$. Then 
$
u \widehat{u}= D_q(u) \II
$.
\end{prop}
\begin{proof}
The lengthy proof is in Appendix~\ref{app:A}.
\end{proof}

As  a direct consequence of this proposition (and of the uniqueness of the antipode), in the quotient algebra $\soq{3}=\aoq{3}/ \langle D_q(u)-1 \rangle$ we can then identify  
the matrix $\widehat{u}= (\hu{j}{k})_{j,k=1,2,3}$ of cofactors with the antipode matrix: 
\beq\label{S=cofactor}
S(u)= 
\begin{pmatrix}
\u{3}{3}  & q^{-\half}  \u{2}{3}  & q^{-1}\u{1}{3} 
\\
q^{\half}  \u{3}{2}  & \u{2}{2}  & q^{-\half} \u{1}{2} 
\\
q~ \u{3}{1}  & q^{\half}  \u{2}{1}  & \u{1}{1}  
\end{pmatrix}  =
S(u) u \widehat{u}  = \widehat{u}=
 ~^t\mathsf{cof}(u) \, .
\eeq
In particular, for later use in the  study of coinvariant elements in \S \ref{sec:fibr-princ} below,  we observe we have the following identification among elements of the second column of the matrix $u$ (or second raw of the matrix $S(u)$) and  the corresponding cofactors:
\begin{align}\label{2ccof} 
q^{-\half} \u{1}{2} &= - \u{1}{1}\u{2}{3} + \u{1}{3} \u{2}{1}  -  (q^\half -q^{-\half}) \u{1}{2}\u{2}{2}
\\ \nn 
&=q^{-1} \u{2}{1}\u{1}{3} - q^{-1}\u{2}{3} \u{1}{1}  +  q^{-1} (q^\half -q^{-\half}) \u{2}{2}\u{1}{2} 
\\ \nn
\u{2}{2}  &= \u{1}{1}\u{3}{3}-  \u{1}{3}\u{3}{1}  + (q^\half -q^{-\half}) \u{1}{2}\u{3}{2} 
\\ \nn
&=
- \u{3}{1}\u{1}{3}+  \u{3}{3}\u{1}{1}  - (q^\half -q^{-\half}) \u{3}{2}\u{1}{2} \\ \nn
&= (q^\half -q^{-\half})^{-1} (\u{2}{1}\u{2}{3}-  \u{2}{3}\u{2}{1} + (q^\half -q^{-\half}) \u{2}{2}\u{2}{2} ) 
\\ \nn
q^{\half} \u{3}{2} &= -q  \u{2}{1}\u{3}{3} + q \u{2}{3} \u{3}{1}  -q   (q^\half -q^{-\half}) \u{2}{2}\u{3}{2} 
\\ \nn
&= \u{3}{1}\u{2}{3} - \u{3}{3} \u{2}{1}  +   (q^\half -q^{-\half}) \u{3}{2}\u{2}{2}  \; .
\end{align}

\subsection{Two real forms of $\soq{3}$}\label{sec:*struct}
As already mentioned above for general $N$,  the Hopf algebra $\aoq{3}$ can be equipped with different real structures \eqref{*-struct1} or \eqref{*-struct2}, depending on the deformation parameter $q$:
\begin{eqnarray}
(\u{j}{k})^*= S(\u{k}{j})= 
q^{\rho_j-\rho_k} \u{j'}{k'}, &&\mbox{ for }q \in \IR;  \label{*O3R}
\\
(\u{j}{k})^*= \u{j}{k}, &&\mbox{ for } |q|=1 ~ \label{*O3C} .
\end{eqnarray}
These lead to the Hopf $*$-algebras 
 $\cO(O_q(3, \IR))$  for $q \in \IR$ and  $\cO(O_q(1,2))$  for $|q|=1$.

Moreover, by direct verification, it is easy to check that
\begin{lem}
The exterior algebra  
$\Lambda_q(V)$ in \eqref{comm-e} is a $*$-algebra 
  with involution $*:\Lambda_q(V) \to \Lambda_q(V)$ defined on generators $e_k$, $k=1,2,3$ by
\begin{eqnarray}
&e_k^*=q^{\rho_k} e_{k'} & \mbox{ for } q \in \IR \, , \label{*extR}
\\
&e_k^*= e_{k} & \mbox{ for } |q|=1 \label{*extC} \, .
\end{eqnarray}
Then, for $q \in \IR$, respectively $ |q|=1$,  the coaction $\rho: \Lambda_q(V) \to\aoq{3} \ot \Lambda_q(V)$, $e_k \mapsto \sum_j u_{kj} \ot e_j$ is a $*$-map with respect to the 
 $*$-structures on $\aoq{3}$ defined in \eqref{*O3R}, respectively \eqref{*O3C}.
\end{lem}

\begin{lem}
For $q \in \IR$, respectively $|q|=1$, the quantum determinant $D_q(u)$ in \eqref{det} is real with respect to the $*$-structures on $\aoq{3}$ defined in \eqref{*O3R}, respectively \eqref{*O3C}.
\end{lem}
\begin{proof}
For each three-form $\xi \in \Lambda_q(V)$, from 
$\rho(\xi) = D_q(u) \ot \xi$, it follows that $D_q(u)^* \ot \xi^* = \rho(\xi)^* = \rho(\xi^*) = D_q(u) \ot \xi^*$ and therefore that the quantum determinant is real: $D_q(u)^*=D_q(u)$. 
(Alternatively, the Lemma can be proved by comparing $D_q(u)^*$ computed from \eqref{det} 
with the formula for $D_q(u)$ given by  the expansion of the quantum determinant with respect to the third row.)
\end{proof}
It follows that   
$\langle D_q(u)-1 \rangle$ is a $*$-ideal. 
For $q \in \IR$, we denote by  $\cO(SO_q(3, \IR))$ the quotient Hopf $*$-algebra $\cO(O_q(3, \IR))/\langle D_q(u)-1 \rangle$ with $*$-structure inherited from that of $\cO(O_q(3, \IR))$ in \eqref{*O3R}. While we denote by $\cO(SO_q(1,2))$
the quotient Hopf $*$-algebra $\cO(O_q(1,2))/\langle D_q(u)-1 \rangle$  with $*$-structure inherited from that of $\cO(O_q(1,2))$ in \eqref{*O3C}.

\subsection{The double covering of $SO_{q}(3)$}\label{sec:dc}
Classically, the Lie group $SL(2)$ is a double covering of $SO(3)$. The quantum analogue of this fact was proven in \cite{dij} where it was shown the existence of a 
 Hopf algebra isomorphism between the coordinate algebra $\mathcal{O}(SO_{q}(3))$ and the subalgebra of $\mathcal{O}(SL_{\s}(2))$, $\s=q^\half$, made of coinvariant elements for the coaction of the group algebra $\IC \IZ_2$ of $\IZ_2$.  If we denote by $a,b,c,d$ the generators of $\mathcal{O}(SL_{\s}(2))$,  the defining matrix and commutation relations are given by 
\beq\label{SLs2}
v:=\begin{pmatrix}
a & b \\c & d
\end{pmatrix}  \quad 
\begin{array}{lll}
ab=\s~ba \, ,\quad &ac= \s~ca \, ,\quad &bd= \s ~db \, ,
\\
cd= \s ~dc \, ,\quad & bc= cb\, ,\quad  & ad= da +(\s-\s^{-1}) bc
\end{array}
\eeq
 with moreover $ad- \s \, bc=1$. In matrix notation, $\mathcal{O}(SL_{\s}(2))$
has coproduct $\Delta (v)=v \ot v$, counit $\varepsilon(v)=\II$ and antipode $S(v)= \begin{pmatrix}
d & -\s^{-1}b \\- \s ~c & a
\end{pmatrix}$.
Denoting  by $u$ the defining matrix of  $\mathcal{O}(SO_{q}(3))$, as before, the Hopf algebra isomorphism is given by
\beq\label{dc}
u \mapsto
\begin{pmatrix}
a^2 & (1+\s^{2})^\half ba & -b^2
\\
(1+\s^{2})^\half ca & 1+ (\s +\s^{-1})bc & -(1+\s^{2})^\half db
\\
-c^2 & - (1+\s^{2})^\half dc & d^2
\end{pmatrix}.
\eeq

\section{The orthogonal $2$-sphere and hyperboloid}\label{sec:qvs}

As mentioned in \S \ref{sec:q-det} for the general case, associated with the quantum group $\soq{3}$ there is  a quantum vector space $\qvs$. It is defined, via the the $q$-symmetrizer matrix $P_-$ in the decomposition  \eqref{decr} of the $R$-matrix, as the free algebra generated by three elements $\x{k}$, $k=1,2,3,$ modulo an ideal of relations:
$$
\aqvs := \IC \langle \x{k}\rangle /{\langle (P_-)^{jl}_{mn} \x{m}\x{n} \rangle} \; .
$$
Explicitly, with the $R$-matrix  in \eqref{Rmat3}, the algebra relations are given by
\beq\label{rel_C3}
\x{2}\x{1}= q^{-1}  \x{1}\x{2} \; , \quad
\x{3}\x{2}  = q^{-1}  \x{2}\x{3} \; , \quad
\x{3}\x{1}  =  \x{1}\x{3}+ (q^\half - q^{-\half})  \x{2}^2 \; .
\eeq
By construction $\aqvs$ carries a left coaction which is an algebra homomorphism and is given by
$$
\rho: \aqvs \to \asoq{3} \ot \aqvs \; , \quad \x{k} \mapsto \u{k}{m} \ot \x{m} \; .
$$ 
It is easy to see that the quadratic element
\beq
r:= q^{-\half}  \x{1}\x{3}+\x{2}^2 + q^\half \x{3}\x{1} 
\eeq
belongs to the centre of the algebra and the coaction $\rho$ descends to a coaction on the quotient algebra 
$\aqvs/ \langle r-1 \rangle$. 

There are again two $*$-structures as the ones in \eqref{*y} making $\aqvs$ a $*$-algebra. 
For $q \in \IR$, the involution is $\x{k}^*=q^{\rho_k} \x{k'}$, or explicitly,
\beq
\x{1}^*=q^{\half} \x{3} \; , \quad 
\x{2}^*= \x{2} \; , \quad 
\x{3}^*=q^{-\half} \x{1} \; , 
\eeq
while for $|q|=1$ the algebra $\aqvs$ is a $*$-algebra for $\x{k}^*=\x{k}$.

For both choices of $q$ the central element $r$ is real, $r^*=r$; thus the
quotient algebras $\aqvs/ \langle r-1 \rangle$ are left comodules $*$-algebra 
for the corresponding Hopf $*$-algebras obtained from $\asoq{3}$, that is 
$\cO(SO_q(3,\IR))$ and $\cO(SO_q(1,2))$. 

In order to understand the geometry of the quantum spaces described by the $*$-algebras $\aqvs/ \langle r-1 \rangle$ one introduces cartesian coordinates. 
Consider the following generators:
\beq
\X{1}:= \mu \, \ii \, \frac{1}{\sqrt{2}} \left(- \alpha \x{1} +\beta \x{3} \right)
\, ,\quad 
\X{2}:=\gamma \, \x{2}
\, ,\quad 
\X{3}:=\frac{1}{\sqrt{2}} \left(\alpha \x{1} + \beta \x{3} \right)
\eeq
with $\alpha,\beta,\gamma, \mu \in \IC$ such that
$$
\alpha \beta = \frac{1}{2} (q^\half + q^{-\half}) \,  \, , \qquad 
\gamma^2 = \half (q + q^{-1}) \, , \qquad 
\mu = \left\{ \begin{array}{l} 1 \; \mbox{ if } q \in \IR
\\
\\
- \ii  \; \mbox{ if } |q|=1
\end{array} \right. \, .
$$
Provided we choose $\beta= q^\half \bar{\alpha}$ for $q \in \IR$ and $\alpha=\bar{\alpha},\, \beta=\bar{\beta}$ for $|q|=1$,  
the generators $\X{k}$ are real, $\X{k}^*=\X{k}$, for both $*$-structures.
The quadratic identity
$$
q^{-\half}  \x{1}\x{3}+\x{2}^2 + q^\half \x{3}\x{1} =1 ~ ,
$$
 in terms of the real generators $\X{k}$, is easily found to become
\beq
\mu^2 \X{1}^2 + \X{2}^2 +\X{3}^2 =1 \, .
\eeq
This is the equation of a two-sphere if $\mu^2 =1$, or a hyperboloid if $\mu^2 = -1$.

For $q \in \IR$, we denote by $\asphf$ the $*$-algebra $\aqvs/ \langle r-1 \rangle$, the 
coordinate algebra of the quantum Euclidean real unit sphere $\sphf$.
It is a  left comodule $*$-algebra for $\cO(SO_q(3;\IR))$. The sphere $\sphf$ is in fact the 
equatorial  Podle\'s sphere of \cite{po87}. 

For $|q|=1$  we denote by $\ahypf$ the $*$-algebra  $\aqvs/ \langle r-1 \rangle$, the coordinate algebra of the quantum Euclidean hyperboloid $\hypf$, a left comodule $*$-algebra for  $\cO(SO_q(1,2))$.

\subsection{Pre-regular multilinear forms} \label{sec:pmf}

Let $w$ be the linear form on $\IC^3$ with components  
\beq\label{w}
w(v_i,v_j,v_k)=: 
\varepsilon_{ijk} 
\eeq
in the canonical basis $\{v_j, ~j=0,1,2,3\}$ of $\IC^3$,
where $\varepsilon_{ijk} $ is the tensor in \eqref{tensor-epsilon}.

With reference to the theory of pre-regular multilinear forms (see \cite[Def.~2]{dv07}) we have the following result
\begin{lem}
The linear form $w$ is pre-regular, that is   
\begin{enumerate}[(i)]
\item 
there exists an element $T \in GL(3, \IC)$ such that $w$ is $T$-cyclic, i.e. such that
$w(V_1,V_2,V_3)=w(T(V_3),V_1,V_2)$ for all $V_1,V_2,V_3 \in \IC^3$;
\item[~]
\item \quad
if $v \in \IC^3$ is such that  
$w(v,e_j,e_k) = 0$ for all indices $j,k$, then $v=0$. 
\end{enumerate}
\end{lem}
\begin{proof}
Define $T \in GL(3,\IC)$ as the linear transformation $T(v_j)=\mu_j v_j$ for 
$\mu_1=q$, $\mu_2=1$, $\mu_3=q^{-1}$. By direct computation one verifies that $w$ is such that
$w(v_i,v_j,v_k)=w(T(v_i),v_j,v_k)$ on the elements $v_i$ of the basis, for $i,j,k=1,2,3$, being
$\varepsilon_{ijk}  =  \mu_k \varepsilon_{kij} $. 
\end{proof}

\begin{lem}
Let $A(w,2)$ be the quadratic algebra generated by elements $x_i$, $i=1,2,3$,
satisfying the three relations
\beq\label{comm-w}
\sum_{jk} \varepsilon_{ijk} \, x_j x_k =0 , \quad  \mbox{ for } i=1,2,3 \,\, .
\eeq
Then $A(w,2)$ coincides with the algebra $
\aqvs := \IC \langle \x{k}\rangle /{\langle (P_-)^{jl}_{mn} \x{m}\x{n} \rangle}$. 
\end{lem}
\begin{proof}
By direct check, comparing \eqref{comm-w} with relations \eqref{rel_C3}.
\end{proof}

\section{The quantum homogeneous spaces}\label{sec:qhs}

We  already mentioned that part of the definition of quantum orthogonal group requires $N>2$ (and $q^2\not=\pm 1$). There is indeed no quantum group $\oq{2}$ and one rather has that $\oq{2}=O(2)$, the usual orthogonal group in 2-dimensions. Let us better clarify this fact. 
For $N=2$, the defining $R$-matrix of $\aoq{2}$, that we denote by $\widetilde{R}$ to distinguish it from that of $\aoq{3}$,  has a simpler expression. 
Indeed  
formula \eqref{Rmatrix}
$$
\widetilde{R}^{ij}_{mn}=q^{\delta_{ij}-\delta_{ij'}} \delta_{im}\delta_{jn} + (q-q^{-1}) \theta(i-m)
(\delta_{jm}\delta_{in}-q^{-\rho_i-\rho_n}\delta_{ij'}\delta_{nm'})
$$
simplifies for $N=2$. The Heaviside function $\theta$ is non zero only for $i=2, m=1$, but in this case the last summand in $\widetilde{R}$ is zero (being $i'=3-i$ and $\rho_i=0$ for each index i=1,2). Thus $\widetilde{R}$ is 
 diagonal 
and given in matrix form by
\beq\label{Rmat2}
\widetilde{R}= \left(
\begin{array}{cc|cc}
q &  &  &
\\
0 & q^{-1} &  & 
\\
\hline
0 & 0 &q^{-1} & 
\\
0 & 0 &0 & q
\end{array}\right)  .
\eeq

We define $\aoq{2}$ to be the quotient of free algebra $\IC \langle \v{i}{j} \rangle$ generated over $\IC$ by the entries of a matrix $\tilde{u}=(\v{i}{j})$, $i,j=1,2$, modulo the relation \eqref{rtt-matrix}, 
\beq\label{rtt-v}
\widetilde{R}^{ji}_{kl} \v{k}{m} \v{l}{n}=\v{i}{k} \v{j}{l} \widetilde{R}^{lk}_{mn} \quad i,j,k,l=1,2
\eeq
and \eqref{metric} 
\beq\label{metric-v}
\tilde{u} \widetilde{C}\tilde{u}^t \widetilde{C}^{-1} =   \widetilde{C}\tilde{u}^t \widetilde{C}^{-1} \tilde{u}=\II \, , \quad \widetilde{C} = \widetilde{C}^{-1}= \begin{pmatrix}
0 & 1 \\ 1 & 0 \end{pmatrix}.
\eeq
From the matrix \eqref{Rmat2}, relations \eqref{rtt-v} simply read 
$$
 \v{j}{m} \v{i}{n} = q^{\delta_{ij'}-\delta_{ij}+\delta_{mn}-\delta_{mn'}} \v{i}{n} \v{j}{m} 
$$
from which we see that $\aoq{2}$ is a commutative algebra with
$$
 \v{1}{1} \v{2}{2} = \v{2}{2} \v{1}{1}  \, ,\quad \v{1}{2} \v{2}{1} = \v{2}{1} \v{1}{2}
$$
together with  
\footnote{From RTT equations \eqref{rtt-exp} we indeed obtain  identities of the form $\v{1}{1} \v{1}{2} = q^{-2} \v{1}{2} \v{1}{1}$ and 
$\v{1}{1} \v{1}{2} = q^{2} \v{1}{2} \v{1}{1}$, hence concluding by comparison that $\v{1}{1} \v{1}{2} = 0= \v{1}{2} \v{1}{1}$.}
\begin{align}\label{rc-o2}
\v{1}{1} \v{1}{2} = 0= \v{1}{2} \v{1}{1}   \; &; & \quad \v{1}{1} \v{2}{1} = 0 =\v{2}{1} \v{1}{1}
  \\ \nn \quad \v{1}{2} \v{2}{2} = 0= \v{2}{2} \v{1}{2}
  \; &; &   \v{2}{1} \v{2}{2} =  0= \v{2}{2} \v{2}{1} \, .
\end{align}
On the other hand, the metric condition \eqref{metric-v} requires 
\begin{align*}
\v{1}{1}  \v{1}{2} + \v{1}{2}  \v{1}{1}=0 \; ,
\quad \v{2}{1}  \v{2}{2} + \v{2}{2}  \v{2}{1}=0 \; ,
\\
\v{2}{2}  \v{1}{2} + \v{1}{2}  \v{2}{2}=0 \; ,
\quad
\v{2}{1}  \v{1}{1} + \v{1}{1}  \v{2}{1}=0
\end{align*}
as well as
$\v{1}{1}  \v{2}{2} + \v{1}{2}  \v{2}{1}=1$. Thus, excluding zero divisors, either $\v{1}{1} =\v{2}{2} = 0$ or $\v{1}{2}= \v{2}{1} = 0$ in order for \eqref{rc-o2} and $\v{1}{1}  \v{2}{2} + \v{1}{2}  \v{2}{1}=1$
to be both satisfied. The condition $\varepsilon(\v{1}{1}  \v{2}{2} + \v{1}{2}  \v{2}{1})=1$ forces the choice 
$\v{1}{2}= \v{2}{1}=0$.

Obviously, the element $D_q(\tilde{u})= \v{1}{1}  \v{2}{2} $ is central and the quotient algebra 
$\asoq{2}=\aoq{2}/\langle D_q(\tilde{u})-1\rangle$ is 
 just the commutative coordinate algebra of the classical group $SO(2)$ generated by the entries of 
$\tilde{u}= \begin{pmatrix}
\v{1}{1} & 0 \\ 0 & \v{2}{2} 
\end{pmatrix}$ with 
$\v{1}{1}  \v{2}{2}= \v{2}{2}  \v{1}{1} =1$.

The Hopf algebra of $SO(2)$ admits two $*$-structures: 
\begin{enumerate}
\item[] $(\v{k}{k})^*=\v{k}{k}$, for all $k=1,2$, giving the real group $SO(2; \IR)$ \\
\item[] $(\v{1}{1})^*=\v{2}{2}$, giving the real group $ SO(1,1)$ .
\end{enumerate}

\subsection{The quantum principal $SO(2)$-bundle}\label{sec:fibr-princ}
It is known that $SO(2)$ 
is a quantum subgroup of $\soq{3}$ (see e.g. \cite[Thm.~3.5]{po95}.  
Indeed, it is easily shown that $I:=\langle \u{i}{j} | i \neq j \rangle$ is a Hopf ideal in $\asoq{3}$.    
The quotient Hopf algebra $\asoq{3}/I$ is generated by the elements $\v{i}{j}:= \pi(\u{i}{j})$, for $\pi$ the quotient map $\pi: \asoq{3} \to \asoq{3}/I$, 
and thus has just three generators $\v{i}{i}$, $i=1,2,3$.
Their commutation relations are obtained via the projection $\pi$ from those of $\asoq{3}$.
From the equation \eqref{rtt-exp} we simply get
$$
\v{j}{j} \v{k}{k}= \v{k}{k}\v{j}{j} \; , \; \forall j,k=1,2,3. 
$$ 
In addition, the metric condition \eqref{metric} requires that $\v{1}{1} \v{3}{3}=1$ and (by using also the counit $\varepsilon$) that $\v{2}{2} = 1$.
Thus the Hopf algebra $\asoq{3}/I$ is a copy of $\cO(SO(2))$, that realises $SO(2)$ as a quantum subgroup of $\soq{3}$.  

The construction is compatible with both $*$-structures of $\asoq{3}$ for the two cases, $q \in \IR$ or $|q|=1$. That is, the ideal $I$ is a $*$-ideal with respect to both of them and the quotient spaces are Hopf $*$-algebras. 
In particular,  $\asoq{3}/I$ is isomorphic to the $*$-algebra $\cO(SO(2; \IR))$ in the case  $q\in \IR$ and 
to $\cO(SO(1,1; \IR))$ in the case $|q|=1$.

By a general construction, there is then a natural (right) coaction  of $SO(2)$ on $\soq{3}$ given by restriction of the coproduct, written in matrix notation as
\begin{align}\label{coact-right}
\delta = (\id \ot \pi) \Delta: \asoq{3} & \to \asoq{3} \ot \cO(SO(2)) 
\nn \\
\begin{pmatrix}
\u{1}{1}  &    \u{1}{2}  & \u{1}{3} 
\\
 \u{2}{1}  & \u{2}{2}  & \u{2}{3} 
\\
\u{3}{1}  &  \u{3}{2}  & \u{3}{3}  
\end{pmatrix} &\mapsto 
\begin{pmatrix}
\u{1}{1}  &    \u{1}{2}  & \u{1}{3} 
\\
 \u{2}{1}  & \u{2}{2}  & \u{2}{3} 
\\
\u{3}{1}  &  \u{3}{2}  & \u{3}{3}  
\end{pmatrix} \ot \begin{pmatrix} 
z  & 0  & 0
\\
0 & 1& 0
\\
0& 0  & z^{-1}  
\end{pmatrix} ,
\end{align}
where we set $z:=\v{1}{1}$.

Since $\delta(\u{i}{j})= \sum_k \u{i}{k} \ot \pi(\u{k}{j})= \u{i}{j} \ot \pi(\u{j}{j})$, it is clear that the  elements $b \in \asoq{3}$ which are coinvariant for the coaction,  $\delta(b)= b \ot 1$, are given in degree one by the span of the elements in the second column of the defining matrix $u$ of $\asoq{3}$ and, in addition, in degree two by the span of  products of any element of the first column with any one of the third, $\u{i}{1}\u{j}{3}$ or
$\u{i}{3}\u{j}{1}$ for  indices $i,j=1,2,3$. Nevertheless, we next show that all the elements $\u{i}{1}\u{j}{3}$ and $\u{i}{3}\u{j}{1}$ indeed belong to the span of those of the second column. 
\begin{prop}\label{prop:B}
The subalgebra 
$$
B:=\asoq{3}^{\text{co}\, \cO(SO(2))} = \{b \in \asoq{3} ~|~ \delta(b)= b \ot 1 \}
$$ 
of coinvariant elements of $\asoq{3}$ with respect to the coaction $\delta$ of $\cO(SO(2))$ in \eqref{coact-right}, is generated by the three elements $\u{i}{2}$, for $i=1,2,3$.
\end{prop}
\begin{proof}
We show that the elements $\u{i}{3}\u{j}{1}$ and $\u{i}{1}\u{j}{3}$ can be written as polynomials in the elements of the second column. By taking $m=3,~ n=1$ in \eqref{rtt-exp}, we obtain
\beq\label{cr13}
 q^{-1}  \u{i}{1} \u{j}{3} 
=
q^{\delta_{ij}-\delta_{ij'}} \u{j}{3}  \u{i}{1}
+ \lambda \theta(j-i) \u{i}{3}  \u{j}{1} -  \lambda \delta_{ij'} \sum_k \theta(j-k)  q^{-\rho_i-\rho_k}\u{k}{3}  \u{k'}{1}   
\eeq
so it is enough to establish the result for the elements $\u{i}{3}\u{j}{1}$.  (We list nevertheless the expressions of all coinvariant elements in terms of the elements $\u{k}{2}$ in Appendix \ref{app:B}.) 
In the proof we will use the identities
$$
\u{3}{1} \u{1}{3}  = \u{1}{3}  \u{3}{1} 
\, , \quad \u{1}{1} \u{1}{3}  = q^2 \u{1}{3}  \u{1}{1}
\, , \quad \u{3}{1} \u{3}{3}  = q^2 \u{3}{3}  \u{3}{1} ~
$$
obtained from equation \eqref{cr13}, for suitable choices of indices $i,j$,  and
 the identification in \eqref{2ccof} of the elements of the second column of the matrix $u$ as cofactors. We will also use the relations   
$$
u S(u)= 
\begin{pmatrix}
\u{1}{1}  &    \u{1}{2}  & \u{1}{3} 
\\
 \u{2}{1}  & \u{2}{2}  & \u{2}{3} 
\\
\u{3}{1}  &  \u{3}{2}  & \u{3}{3}  
\end{pmatrix}
\begin{pmatrix}
\u{3}{3}  & q^{-\half}  \u{2}{3}  & q^{-1}\u{1}{3} 
\\
q^{\half}  \u{3}{2}  & \u{2}{2}  & q^{-\half} \u{1}{2} 
\\
q~ \u{3}{1}  & q^{\half}  \u{2}{1}  & \u{1}{1}  
\end{pmatrix}
= \II
$$ 
and
$$
S(u)u= \begin{pmatrix}
\u{3}{3}  & q^{-\half}  \u{2}{3}  & q^{-1}\u{1}{3} 
\\
q^{\half}  \u{3}{2}  & \u{2}{2}  & q^{-\half} \u{1}{2} 
\\
q~ \u{3}{1}  & q^{\half}  \u{2}{1}  & \u{1}{1}  
\end{pmatrix} 
\begin{pmatrix}
\u{1}{1}  &    \u{1}{2}  & \u{1}{3} 
\\
 \u{2}{1}  & \u{2}{2}  & \u{2}{3} 
\\
\u{3}{1}  &  \u{3}{2}  & \u{3}{3}  
\end{pmatrix}= \II \, .
$$ 
First, by using $\u{3}{1} \u{1}{3}=  \u{1}{3}\u{3}{1} $ in
the equality
 $(u S(u))_{33}=( S(u)u)_{11}$,
 we get
$$
\u{2}{3} \u{2}{1} =\u{3}{2} \u{1}{2}
  \; .
$$
By comparing the expressions  $(u S(u))_{11}=1$ and $\u{2}{2}=\hu{2}{2}$, we compute 
$$
\u{1}{3}\u{3}{1}= (1+q)^{-1} (1-\u{2}{2}-q^{-\half} \u{1}{2}\u{3}{2}) \, .
$$
Similarly, from $(u S(u))_{12}=0$ and the expression $q^{-\half} \u{1}{2} =\hu{2}{3}$
we obtain 
$$
\u{1}{3}\u{2}{1}= q^{-\half}  (1+q)^{-1} (\u{1}{2}-\u{1}{2}\u{2}{2}) \, ,
$$
while from
$(u S(u))_{13}=0$ and the relation $\u{1}{1} \u{1}{3}  = q^2 \u{1}{3}  \u{1}{1}$ found before,
$$
\u{1}{3}\u{1}{1}= - q^{-\half}  (1+q)^{-1} \u{1}{2}\u{1}{2} \, .
$$
We proceed by comparing $(u S(u))_{21}=0$ and the expression $q^{\half} \u{3}{2} =\hu{2}{1}$
and obtain 
$$
\u{2}{3}\u{3}{1}= q^{-\half}  (1+q)^{-1} (\u{3}{2}-\u{2}{2}\u{3}{2}) \, 
$$
while
$(u S(u))_{23}=0$ and the expression $q^{-\half} \u{2}{1} =\hu{2}{3}$
gives 
$$
\u{2}{3}\u{1}{1}= - q^{\half}  (1+q)^{-1} (\u{1}{2}+q^{-1}\u{2}{2}\u{1}{2}) \, .
$$
Then, from $(u S(u))_{31}=0$ and the relation $\u{3}{1} \u{3}{3}  = q^2 \u{3}{3}  \u{3}{1}$
we promptly get
$$
\u{3}{3}\u{3}{1}= -q^{-\half}  (1+q)^{-1} \u{3}{2}\u{3}{2}\, .
$$
From
$(u S(u))_{32}=0$ and the expression $q^{\half} \u{3}{2} =\hu{2}{1}$
we have 
$$
\u{3}{3}\u{2}{1}= -  (1+q)^{-1} ( q^{\half} \u{3}{2}+ q^{-\half}\u{3}{2}\u{2}{2}) \, .
$$
Finally, from  $(u S(u))_{33}=1$ and the equality $\hu{2}{2}=\u{2}{2}$ we obtain the last required relation
\begin{equation*}
\u{3}{3}\u{1}{1} = (1+q)^{-1} (q+\u{2}{2} - q^{-\half}\u{3}{2}\u{1}{2}) \, .
\qedhere \end{equation*}
\end{proof}

The commutation relations among the generators $\u{k}{2}$ of the subalgebra $B$ of coinvariants are obtained from equations \eqref{rtt-exp}, for $m=n=2$,
\begin{align*}
q^{\delta_{ij}-\delta_{ij'}} \u{j}{2}  \u{i}{2}  = &~ (1- \lambda  \theta(j-i) ) \u{i}{2} \u{j}{2} 
+ 
  \lambda \delta_{ij'} \sum_k \theta(j-k)  q^{-\rho_i-\rho_k}\u{k}{2}  \u{k'}{2}   
- \lambda   q^{-\half}\u{i}{1}  \u{j}{3} ~ \nn
\end{align*}
by substituting the explicit expression of the elements 
$\u{i}{1}  \u{j}{3}$ in terms of the elements $\u{k}{2}$ (as given in Appendix \ref{app:B}).
They are given by
\begin{align}\label{comm-rel-u2}
&\u{3}{2} \u{2}{2} = q^{-1} \u{2}{2} \u{3}{2} + (1-q^{-1}) \u{3}{2}  \, , \quad
\u{2}{2} \u{1}{2} = q^{-1} \u{1}{2} \u{2}{2} + (1-q^{-1}) \u{1}{2} \, , \nn
\\
&\u{3}{2} \u{1}{2} = q^{-2} \u{1}{2} \u{3}{2} + q^{-\half}(1-q^{-1}) (1-\u{2}{2}) \, .
\end{align}
Moreover from condition $(S(u)u)_{22}=1$ we also obtain
\beq
q^{\half} \u{3}{2} \u{1}{2} + q^{-\half}\u{1}{2} \u{3}{2} +(\u{2}{2}-1)(\u{2}{2}+1)=0\, .
\eeq

We will analyse the geometry of  $B$ as a quantum $*$-algebra in \S \ref{sec:*qs} below. 
Before we do that, we study the bundle structure of the quantum homogeneous space $B$.

\begin{prop}
The algebra extension $B=\asoq{3}^{co\, \cO(SO(2))} \subset \asoq{3}$ is Hopf-Galois, that is 
the canonical map
$$
\chi: \asoq{3} \ot_{B} \asoq{3} \to \asoq{3} \ot \cO(SO(2))  , \quad a' \ot a \mapsto a' \delta(a) 
$$ 
is bijective. 
\end{prop}
\begin{proof}
We prove the statement by showing that the total space algebra $\asoq{3}$ is strongly graded 
(see Thm. 4.3 and Prop. 4.6 of \cite{AKL16}). 
We assign degree $+1$ to the elements of the first column of the defining matrix $u$, degree $-1$ to the elements of the third column of the matrix $u$ and degree $0$ to the elements of the central column of the matrix $u$. Let us denote $\mathcal{E}_{\pm 1}$ the collection of all degree $\pm$ elements respectively in $\asoq{3}$. Clearly they are modules over $\mathcal{E}_0=B$; \emph{a posteriori} these are shown to be finitely generated and projective over $B$ (see \cite[Cor. I.3.3]{NvO82}).   
 
 In the notation of \cite{AKL16} we have two sequences of elements in $\mathcal{E}_{+1}$:
\beq\label{xb}
 \{\xi_j \}_{j=1}^{3} = (\u{1}{1}, \u{2}{1}, \u{3}{1})  , \qquad  
 \{\beta_j \}_{j=1}^{3} = (q \u{3}{1}, q^{\half} \u{2}{1}, \u{1}{1})     
\eeq
 and two sequences of elements in $\mathcal{E}_{-1}$:
\beq\label{ea}
  \{\eta_j \}_{j=1}^{3} = (\u{3}{3}, q^{-\half}  \u{2}{3}, q^{-1}\u{1}{3} )  , \qquad  
 \{\alpha_j \}_{j=1}^{3} = (\u{1}{3}, \u{2}{3}, \u{3}{3} ) \, .
 \eeq
These are such that 
\beq\label{11}
\sum_{j=1}^{3} \eta_j  \xi_j = (S(u)u)_{11}=\u{3}{3} \u{1}{1} + q^{- \half} \u{2}{3} \u{2}{1} + q^{-1} \u{1}{3} \u{3}{1} = 1
\eeq 
and 
\beq\label{22}
\sum_{j=1}^{3} \beta_j  \alpha_j = (S(u)u)_{33}=q \u{3}{1} \u{1}{3} + q^{ \half} \u{2}{1} \u{2}{3} +  \u{1}{1} \u{3}{3} = 1 \, .
\eeq
 The inverse $\chi^{-1} : \asoq{3} \ot \cO(SO(2)) \to \asoq{3} \ot_B \asoq{3}$ of the canonical map, 
 by the general theory of \cite{AKL16}, is then given by 
\beq\label{chiinv}
\chi^{-1} : a \ot z^n \mapsto 
\begin{cases}
\sum_{J \in \{1,2,3\}^n} \, a \, \eta_{j_1} \cdots \eta_{j_n} \ot_B \xi_{j_n} \cdots \xi_{j_1} \, ,
& \mbox{for} \,\,\, n \geq 0  \\ 
a \ot_B 1 \, & \mbox{for} \,\,\, n = 0  \\
\sum_{I \in \{1,2,3\}^{-n}} \, a \, \beta_{i_1} \cdots \beta_{i_{-n}} \ot_B \alpha_{i_{-n}} \cdots \alpha_{i_1} \, ,
& \mbox{for} \,\,\,n \leq 0
\end{cases} \,\, .
\eeq
For the convenience of the reader we recall here the proof. If $n\geq 0$, 
\begin{align*}
\chi \circ \chi^{-1}( 1\ot z^n) & = 
\chi( \sum_{J \in \{1,2,3\}^n} \, a \, \eta_{j_1} \cdots \eta_{j_n} \ot_B \xi_{j_n} \cdots \xi_{j_1} ) \\
& = \sum_{J \in \{1,2,3\}^n} \, a \, \eta_{j_1} \cdots \eta_{j_n} \, \xi_{j_n} \cdots \xi_{j_1} 
 \ot z^n= 1 \ot z^n,  
\end{align*}
using \eqref{11} on all indices from $j_n$ to $j_1$ one after the other. Conversely, if $a \in \asoq{3}$ is of degree $n$, one has $\delta(a) = a \ot z^n$ and thus 
\begin{align*}
\chi^{-1} \circ \chi(1 \ot_B a) & = \chi^{-1}(a \ot z^n) = 
\sum_{J \in \{1,2,3\}^n} \, a \, \eta_{j_1} \cdots \eta_{j_n} \ot_B \xi_{j_n} \cdots \xi_{j_1} \\
& =
\sum_{J \in \{1,2,3\}^n} \, 1 \ot_B a \, \eta_{j_1} \cdots \eta_{j_n} \xi_{j_n} \cdots \xi_{j_1} = 1 \ot_B a
\end{align*}
using the fact that $a \, \eta_{j_1} \cdots \eta_{j_n}\in B$, so that it can cross over the balanced tensor product, and again \eqref{11}. One proceeds similarly for $n\leq 0$.
\end{proof}

\subsection{Two $*$-quantum homogeneous spaces of $\asoq{3}$}\label{sec:*qs}

We rename $\y{k}:=\u{k}{2}$, $k=1,2,3$ the generators of the subalgebra $B$ 
of coinvariant elements of $\asoq{3}$. They have commutation relations  \eqref{comm-rel-u2}:
\begin{align}\label{comm-rel-y}
&\y{3} (\y{2}-1) = q^{-1} (\y{2}-1) \y{3}   \, , \quad
 \y{1} (\y{2}-1) = q (\y{2}-1)  \y{1} \, , \nn
\\
&
\y{3} \y{1} = q^{-2} \y{1} \y{3} + q^{-\frac{3}{2}}(1-q) (\y{2}-1)
\end{align}
and satisfy the quadratic condition
\beq \label{quadric-y}
 q^{-\half}\y{1} \y{3} +q^{\half} \y{3} \y{1} +\y{2}^2=1\, .
\eeq
This, with the last equation in \eqref{comm-rel-y}, can also be written as
\begin{align}\label{1331}
(q^\half+q^{-\half}) \y{1} \y{3} & = (1 - \y{2}) (1 + q \y{2})  \nn \\ 
(q^\half+q^{-\half}) \y{3} \y{1} & = (1 - \y{2}) (1 + q^{-1} \y{2}) .
\end{align}

It is easy to see that  the coaction map $\delta$ in \eqref{coact-right} is a $*$-map, that is
$\delta(\u{k}{m}^*)= \delta(\u{k}{m})^*$, for both $q \in \IR$ 
and $|q|=1$ and corresponding $*$-structures. 
Hence $B$ is a $*$-algebra as well with $*$-structures inherited by those of $\asoq{3}$ and given on the generators $\y{k}$ by 
\begin{align}\label{*y}
&\mbox{ for } q \in \IR: \qquad (\y{1})^*= q^{\half} \y{3} \; ; \quad (\y{2})^*= \y{2} \; ; \quad  (\y{3})^*= q^{-\half} \y{1} \; , \nn \\
&\mbox{ for } |q|=1: \qquad (\y{k})^*=\y{k} \,, \quad k=1,2,3 \; .
\end{align}
Moreover, the $*$-algebra $B$ is made of coinvariant elements of the corresponding real group by a suitable real subgroup.

 For $q \in \IR$, we denote   $\asph$ the $*$-algebra $B$ of coinvariant elements of $\cO(SO_q(3,\IR))$ with respect to the coaction of its quantum subgroup $\cO(SO(2,\IR))$. We call $\asph$ (the algebra of coordinate functions of) the quantum 
(Grassmannian) sphere $\sph$. 
In the limit $q=1$ it reduces to the coordinate algebra over the Grassmannian $Gr(1,3)\simeq SO(3)/SO(2) \simeq S^2$ of oriented lines in $\IR^3$.   
In fact, the sphere $\sph$ is isomorphic to the standard  
  Podle\'s sphere $\sphf$ of \cite{po87}.

For $|q|=1$, we denote $\ahyp$ the $*$-algebra $B$ of coinvariant elements of $\cO(SO_q(1,2))$ with respect to the coaction of its quantum subgroup $\cO(SO(1,1))$. We call $\ahyp$ (the algebra of coordinate functions of) the quantum hyperboloid $\hyp$. In the limit $q=1$ it reduces to the coordinate algebra over the hyperboloid.

Again, as in \S \ref{sec:qvs}, the reason for the names and the nature of the spaces above is made evident when using cartesian coordinates. Let us make the following change of generators:
\beq
\Y{1}:= \mu \, \ii \, \frac{1}{\sqrt{2}} \left(- \alpha \y{1} + \beta \y{3} \right)
\, , \quad 
\Y{2}:=\y{2}
\, , \quad 
\Y{3}:=\frac{1}{\sqrt{2}} \left(\alpha\y{1} + \beta \y{3} \right)
\eeq
with $\alpha, \beta, \mu \in \IC$ such that
$$
\alpha \beta=q^\half \frac{  (1+q)}{(1+q^2)} \, , \qquad 
\mu= \left\{ \begin{array}{l} 1 \; \mbox{ if } q \in \IR
\\
\\
-\ii \; \mbox{ if } |q|=1
\end{array} \right. \, .
$$
Notice that $(\alpha \beta)^*=\alpha \beta$ for  both choices of $q$. Provided we choose $\beta= q^\half \bar{\alpha}$ for $q \in \IR$ and $\alpha=\bar{\alpha},\, \beta=\bar{\beta}$ for $|q|=1$, for both $*$-structures in \eqref{*y}, the generators $\Y{k}$ are real, 
$$
(\Y{k})^*=\Y{k} , \qquad k=1,2,3.
$$
Using relations \eqref{1331}, we   compute
$$ \mu^2 \Y{1}^2  +\Y{3}^2= \frac{q^\half  (1+q)}{(1+q^2)} (\y{1}\y{3} +
\y{3}\y{1})= - \frac{1}{(1+q^2)} \left[  (1+q^2) \y{2}^2 -(1-q)^2 \y{2} -2q \right] 
$$ 
and thus in terms of the real generators $\Y{k}$  the quadratic condition \eqref{quadric-y} reads 
\beq
\mu^2 \Y{1}^2 + \Y{2}^2 +\Y{3}^2 - \frac{ (1-q)^2}{1+q^2} \Y{2}= \frac{2q}{1+q^2} \, .
\eeq
In the classical limit $q \rightarrow 1$ this reduces to
$$
\mu^2 \Y{1}^2 + \Y{2}^2 +\Y{3}^2 = 1
$$
which is a two-sphere if $\mu^2 =1$, or a hyperboloid if $\mu^2 = -1$. \medskip

Let us finally observe (for future use in \S \ref{sec:*structU}) that by construction the subalgebra $B$ also carries a left coaction of $\asoq{3}$ given by the restriction of the coproduct of $\asoq{3}$ to the elements $\u{k}{2}$ generating $B$: the map
\beq\label{coact-left}
\rho=\Delta_|{_B}:  B \to \asoq{3} \ot B \; , \quad \u{k}{2} \mapsto \sum_m \u{k}{m} \ot \u{m}{2}
\eeq
makes $B$ a left $\asoq{3}$-comodule algebra.
The coaction map $\rho$ in \eqref{coact-left} is  a $*$-map for both values of $q$ and thus $B$ is a comodule $*$-algebra, or quantum $*$-algebra, with respect to the corresponding
real forms of $\asoq{3}$.

\subsection{Line bundles}

In general, given a right $H$-comodule algebra $A$ with coaction
$
\delta: A \to A \ot H$, $\delta(a) = \zero{a} \ot \one{a}
$ 
and a left    $H$-comodule $V$  with coaction
$
\gamma: V \to H \ot V$, $\gamma(v) =  \mone{v} \ot \zero{v}$,
sections of the vector bundle associated with the corepresentation  $\gamma$ can be identified with linear maps
 $\phi: V \to A$
 which are $H$-equivariant
\beq\label{seeq}
\zero{\phi(v)} \ot \one{\phi(v)} = \phi(\zero{v}) \ot S(\mone{v}) \, .
\eeq
The collection $\mathcal{E}$ of such maps is a left $B$-module for $B\subseteq A$ the subalgebra of coinvariant elements for the $H$-coaction. 

For the $H=\cO(SO(2))$ Hopf-Galois extension $B=\asoq{3}^{co\, \cO(SO(2))} \subset \asoq{3}$ irreducible corepresentations of $\cO(SO(2))$, which are one dimensional and labelled by an integer, will yield line bundles. Consider any such a corepresentation
\beq\label{co-equiv}
\gamma_n : \IC \to \cO(SO(2)) \ot \IC, \qquad \gamma_n(1) = 1 \ot z^{-n}  
\eeq 
fon any integer $n$. From the coaction \eqref{coact-right} the first column of the matrix $u$ will transform by 
$z^{-n}$ while the last column will transform by $z^{n}$. Thus, using the generators \eqref{xb} and \eqref{ea}, a set of generators of the corresponding $B$-module $\mathcal{E}_n$ of sections is given by
\begin{align}
& \xi_J := \xi_{j_n} \cdots \xi_{j_1} , \qquad J = (j_1, \cdots , j_n) \in \{1,2,3\}^n \qquad \mbox{for} \quad n \geq 0 \nn \\
& \alpha_I := \alpha_{i_{-n}} \cdots \alpha_{i_1} , \qquad I = (i_1, \cdots , i_n) \in \{1,2,3\}^{-n} \qquad \mbox{for} \quad n \leq 0 \, .
\end{align}
Indeed, for $n \geq 0$, one finds that
$$
\delta(\xi_J) = \zero{(\xi_{j_n} \cdots \xi_{j_1})} 
\ot \one{(\xi_{j_n} \cdots \xi_{j_1})} = (\xi_{j_n} \cdots \xi_{j_1}) \ot z^n 
= (\xi_{j_n} \cdots \xi_{j_1}) \ot S(z^{-n}),
$$
thus fulfilling condition \eqref{co-equiv}. The case for negative $n$ works similarly.  
The modules $\mathcal{E}_n$ are line bundles of even degree $2n$. To see this, one 
finds suitable idempotents  $p_n $ in $\M_{| 2n | +1}(B)$ and identifies 
$\mathcal{E}_n \simeq B^{| n | +1} p_n$ as left $B$-modules.  

The idempotents $p_n$ are representatives of classes in the
K-theory of $B$, $[p_n]\in K_0(B)$. One computes the
corresponding rank and degree  by pairing them with non-trivial
elements in the dual K-homology, that is  with (the class of) non-trivial
Fredholm modules $[\mu]\in  K^0(B)$. For this, one first calculates
the corresponding Chern characters in the cyclic homology
$\chern_\bullet(p_n) \in \mathrm{HC}_\bullet(B)$ and cyclic
cohomology $\chern^\bullet(\mu)\in \mathrm{HC}^\bullet(B)$
respectively, and then uses the pairing between cyclic homology and
cohomology.

The Chern character of the idempotents $p_n$ has a non-trivial
component in degree zero $\chern_0(p_n)\in \mathrm{HC}_0(B)$ given
simply by a (partial) matrix trace
$\chern_0(p_n) := \tr (p_n)$
and thus $\chern_0(p_n)\in B$. Dually, one needs a cyclic zero-cocycle,
i.e. a trace on $B$. 
There are indeed two such traces. One is the restriction of the counit $\varepsilon$ of $\asoq{3}$ to $B \subset \asoq{3}$; this computes the rank of the bundle.  On generators is given by
\beq\label{co-unsp}
\varepsilon(y_1)= \varepsilon(y_3)=0, \quad \varepsilon(y_2) = 1. 
\eeq
The second `singular' trace was obtained in \cite{MNW} and it is a
trace on $B  / \IC$, that is it vanishes on $\IC \subset B$; it computes the degree. 
Its values on generators of $\mathcal{O}(SL_{\s}(2))$ given in \eqref{SLs2} was computed 
in \cite{H00} to be (the parameter $q$ there is mapped to $\s^{-1}$ here), 
$$
\mu\left((b \,c)^k\right) = (-1)^k \frac{\s^{-k}}{1-\s^{-2k}} = (-1)^k \frac{ q^{\frac{1}{2} k}}{q^{k} - 1} , \qquad k>0 \ .
$$
Using the identification \eqref{dc} this can be translated to the generator $y_2$ of the algebra $\asphf$:  
from $y_2 = 1+ (q^{\half} + q^{-\half}) bc$ one computes that 
\beq\label{muy}
\mu((y_2-1)^k) = (-1)^k \frac{(q+1)^k}{q^k-1} . 
\eeq

Let us first illustrate the above for the lowest values $n=\pm 1$. In these cases a collection of generators for the modules of sections is given by $(\u{1}{1}, \u{2}{1}, \u{3}{1})$ and $(\u{3}{3}, \u{2}{3}, \u{1}{3})$ respectively. 
 The corresponding idempotents are the matrices
\beq\label{p+}
p_{+1}:= \begin{pmatrix}
\u{1}{1}  
\\
 \u{2}{1}   
\\
\u{3}{1}  
\end{pmatrix}
\begin{pmatrix}
\u{3}{3}, & q^{-\half}  \u{2}{3}, & q^{-1}\u{1}{3} 
\end{pmatrix} \; 
\eeq
and 
\beq\label{p-}
p_{-1}:= \begin{pmatrix}
 \u{3}{3} 
\\
 \u{2}{3} 
\\
 \u{1}{3}  
\end{pmatrix}
\begin{pmatrix}
\u{1}{1}, & q^{\half} \u{2}{1}, & q\, \u{3}{1}     
\end{pmatrix} .
\eeq
Since $p_{+1}$ has entries $(p_{+1})_{ij}= u_{i1}  S(u)_{1j}$, the identity \eqref{11}
 implies that 
$p_{+1}$ is an idempotent $p_{+1}^2= p_{+1}$. 
Similarly, for   $p_{-1}$ of components $(p_{-1})_{ij}= u_{i3}  S(u)_{3j}$, the result $p_{-1}^2= p_{-1}$ follows from 
\eqref{22}.
From Proposition \ref{prop:B} the entries of $p_{\pm 1}$ belong to the subalgebra $B$.
Next, using the list in Appendix \ref{app:B} for quadratic coinvariant elements and the first equality in \eqref{1331}, for the partial trace of these idempotents one computes
\begin{align}\label{trace-p1}
\tr(p_{+1}) &= \u{1}{1} \u{3}{3} + q^{-\half} \u{2}{1} \u{2}{3} + q^{-1} \u{3}{1}\u{1}{3} \nn  \\
&= 1 + (q-1)(y_2-1) + \frac{(q-1)^2}{q+1} (y_2-1)^2 .
\end{align}
Then, using $\varepsilon(y_2) = 1$ one gets
$$
\hs{[\varepsilon]}{[p_{+1}]}:= \varepsilon\left(\chern_0(p_{+1}))\right) = 1 \ .
$$
Finally, using the vanishing of $\mu$ over the scalars and \eqref{muy}
one gets 
\begin{align}\label{ind}
\hs{[\mu]}{[p_{+1}]}
&:= \mu\left(\chern_0(p_{+1})\right) = - (q-1) \frac{(q+1)}{q-1} + (q-1) \frac{(q+1)^2}{q^2-1} \nn \\
&\:=
- (q+1) + (q-1) = -2 \, . 
\end{align}
With a similar computation one gets $\hs{[\varepsilon]}{[p_{-1}]}=1$ and $\hs{[\mu]}{[p_{-1}]}=2$.

For a general $n \geq 0$ consider two vector valued functions of components
$$
\ket{\psi_n}_J := \xi_J = \xi_{j_n} \cdots \xi_{j_1}, \qquad \bra{\phi_n}_J := \eta_J = (\eta_{j_1} \cdots \eta_{j_n}) , 
\qquad J = (j_1, \cdots , j_n) \in \{1,2,3\}^n . 
$$
We have already observed that from \eqref{11} one has
$$
\hs{\phi_n}{\psi_n} = \sum_{J \in \{1,2,3\}^n} \, \eta_{j_1} \cdots \eta_{j_n} \, \xi_{j_n} \cdots \xi_{j_1} 
 \ot z^n = 1 .  
$$
Thus the matrix $p_n = \ket{\psi_n} \bra{\phi_n}$ of components $(p_n)_{KJ} = \xi_J \eta_K$ is an idempotent. 
Similarly, for $n \leq 0$ we take 
$$
\ket{\psi_{-n}}_I := \alpha_I  = \alpha_{i_{-n}} \cdots \alpha_{i_1} , 
\quad \bra{\phi_{-n}}_I := \beta_I = \beta_{i_1} \cdots \beta_{i_{-n}} , 
\quad I = (i_1, \cdots , i_n) \in \{1,2,3\}^{-n} .
$$
and now $\hs{\phi_{-n}}{\psi_{-n}} = 1$ and the idempotent is the matrix $p_{-n} = \ket{\psi_{-n}} \bra{\phi_{-n}}$.

Using an inductive argument and result \eqref{ind}, we show the following.

\begin{prop}\label{prop:pn}
For $n \geq 0$ the modules $\mathcal{E}_n$ are line bundles of even degree $-2n$, that is 
\begin{align}
\hs{[\varepsilon]}{[p_{n}]} = 1 \, \qquad  \hs{[\mu]}{[p_{n}]} = - 2n .  
\end{align}
For $n \leq 0$ one gets $\hs{[\varepsilon]}{[p_{-n}]} = 1$ and positive degree $\hs{[\mu]}{[p_{-n}]} = - 2n$. 
\end{prop}
\begin{proof}
The result rests on a recursion formula for the trace of the idempotents $\tr(p_n)$. 
For
for $n\geq 0$, one finds 
\beq\label{trace-formula}
\tr(p_n) = \sum_J (p_n)_{JJ} = 1 + \sum_{J=1}^{2n}(q+1)^{-J} \, C_J^{(n)} (y_2-1)^J \, , \qquad C_J^{(n)} = \prod_{k=0}^{J-1} (q^{2n-k} -1) . 
\eeq
We prove the formula by induction. We set here $X:=(q+1)^{-1}(y_2-1)$ to simplify notation. 
Firstly, out of the commutation relations \eqref{rtt-exp} one finds
\beq\label{comm-uX}
\u{1}{1} X^J= q^{2J} X^J \u{1}{1} \; , \quad
\u{2}{1} X^J= q^{J} X^J \u{2}{1}\; , \quad
\u{3}{1} X^J=  X^J \u{3}{1}
\eeq
as well as, from the computations in Appendix \ref{app:B}, the identities
\beq\label{id-uX}
 \u{1}{1} \u{3}{3} = 1 + (q +q^2) X + q^3 X^2 \; , \;
q^{-\half} \u{2}{1} \u{2}{3}  =  -(q+1) (X + q X^2) \; , \;
 \u{3}{1}\u{1}{3} =    q X^2      \, .
\eeq
Formula \eqref{trace-formula} is verified for  $n=1$: it is just \eqref{trace-p1}.
Assume it holds for $n$, then
\begin{align*}
\tr(p_{n+1}) 
&=  \u{1}{1} \tr(p_n) \u{3}{3} +  q^{-\half} \u{2}{1} \tr(p_n) \u{2}{3}   + q^{-1} \u{3}{1}\tr(p_n) \u{1}{3} 
\\
&= \tr(p_1) +   \sum_{J=1}^{2n} \, C_J^{(n)} \left( \u{1}{1} X^J \u{3}{3} +  q^{-\half} \u{2}{1} X^J  \u{2}{3}   + q^{-1} \u{3}{1}X^J  \u{1}{3} \right)
\\
&= \tr(p_1) +   \sum_{J=1}^{2n} \, C_J^{(n)} X^J 
\Big( q^{2J} \big(1 + (q +q^2) X + q^3 X^2 \big) - q^J (q+1) (X + q X^2) +X^2  \Big)
\end{align*}
using \eqref{comm-uX} followed by \eqref{id-uX} for the last identity.
Then
\begin{align}\label{traccia}
\tr(p_{n+1})  &= \tr(p_1) +   \sum_{J=1}^{2n} \, C_J^{(n)} X^J \Big(
q^{2J}+(q^{J+1}-1)(q^{J+1}+q^J) X + (q^{J+1}-1)(q^{J+2}-1) X^2
\Big) 
\nn  \\
&= 1 + (q^2-1)X + (q^2-1)(q-1) X^2 
 +   \sum_{J=1}^{2n} \, q^{2J} C_J^{(n)} X^J 
 \\
 & \quad
 +   \sum_{J=2}^{2n+1} \, (q^{J}-1)(q^{J-1} + q^J) C_{J-1}^{(n)} X^J 
 +   \sum_{J=3}^{2n+2} \, (q^{J}-1)(q^{J-1} -1) C_{J-2}^{(n)} X^J  \; .
 \nn
\end{align}
Finally, using properties 
\beq
C_J^{(n)} = (q^{2n+1-J}-1) C_{J-1}^{(n)} \, , \qquad  C_{J+2}^{(n+1)}=  (q^{2n+2}-1)(q^{2n+1} -1) C_J^{(n)}
\eeq
for the coefficients $C_J^{(n)}$, we get
\begin{align*} 
\tr(p_{n+1})  &=  1 + (q^{2n+2}-1)X + \Big( (q^2-1)(q-1) +q^{4} C_{2}^{(n)} + 
(q^2-1)(q + q^2) C_{1}^{(n)} \Big) X^2 
 \\
 & \quad
 +   \sum_{J=3}^{2n} \, \Big( q^{2J} (q^{2n+1-J}-1)  (q^{2n+2-J}-1)  
 +  (q^{J}-1)(q^{J-1} + q^J) (q^{2n+2-J}-1)  
  \\
 & \qquad  \qquad
 +    (q^{J}-1)(q^{J-1} -1) \Big) C_{J-2}^{(n)}  X^J  
  \\
 & \quad
 +  (q^{2n+1}-1)  \Big((q^{2n} + q^{2n+1}) (q-1)
 +   (q^{2n}-1) \Big) C_{2n-1}^{(n)} X^{2n+1}
  \\
 & \quad
 +    (q^{2n+2}-1)(q^{2n+1} -1) C_{2n}^{(n)} X^{2n+2}
 \\
 & =  1 + (q^{2n+2}-1)X +  (q^{2n+2}-1)(q^{2n+1}-1) X^2 
 \\
& \quad
 +   (q^{2n+2}-1)(q^{2n+1}-1)   \sum_{J=3}^{2n} \Big(  C_{J-2}^{(n)}  X^J  
  +  C_{2n-1}^{(n)} X^{2n+1}
 +    C_{2n}^{(n)} X^{2n+2} \Big)
 \\
 & = \sum_{J=1}^{2n+2} \, C_J^{(n+1)} X^J  \; .
\end{align*}
\medskip
Being $\varepsilon(y_2) = 1$, or $\varepsilon(X) = 0$, one immediately gets $\hs{[\varepsilon]}{[p_{+n}]} = 1$.

\noindent
For the computation of the degree we also proceed by induction. 
 From \eqref{muy} one has $\mu(X^J)= (-1)^J \frac{1}{q^J -1}$ from which one deduces
$$
\mu(X^{J+1})= - \frac{q^J -1}{q^{J+1} -1}\mu(X^J)
\; , \quad 
\mu(X^{J+2})= \frac{q^J -1}{q^{J+2} -1}\mu(X^J) \; .
$$
We use these formulas in the first expression in \eqref{traccia} for the trace of $p_{n+1}$:
\begin{align*} 
&\hs{\mu}{\tr(p_{n+1})} = \hs{\mu}{\tr(p_1)} + 
\\
& \qquad +   \sum_{J=1}^{2n} \, C_J^{(n)}  \Big(
q^{2J} X^J+(q^{J+1}-1)(q^{J+1}+q^J) X^{J+1} + (q^{J+1}-1)(q^{J+2}-1) X^{J+2}
\Big) 
   \\
&\qquad\qquad= -2
+ 
  \sum_{J=1}^{2n} \, C_J^{(n)}  \Big(
q^{2J}  -(q^{J}-1)(q^{J+1}+q^J)   + (q^{J+1}-1)(q^{J}-1)  
\Big) \mu(X^J) 
\\
& \qquad\qquad= -2 + \sum_{J=1}^{2n} \, C_J^{(n)}    \mu(X^J) 
\\
& \qquad\qquad =
-2 +\hs{\mu}{\tr(p_{n})}  = -2 (n+1) \; .
\qedhere
\end{align*}
\end{proof}

\begin{rem}
For $q\in \IR$ and $*$-structure \eqref{*O3R}, the idempotent $p_{n}$ is self-adjoint, $p_{n}^*=p_{n}$.
This follows from the fact that $(\ket{\psi_{n}}_J)^* =\bra{\psi_{n}}_J$, for each $J$,
being  $\u{1}{1}^*=\u{3}{3} $, $\u{2}{1}^*=q^{-\half}\u{2}{3} $ and $\u{3}{1}^*=q^{-1}\u{1}{3} $.
We stress that these self-adjoint idempotents are different from the ones used for Podle\'s sphere (see e.g. \cite{H00}), a fact that reflects in a simpler formula for their degree. 
In contrast the idempotents $p_n$ are not self-adjoint for the $*$-structure \eqref{*O3C} when $|q|=1$.
\end{rem}

\section{The Casimir element}
Aiming at the study of laplacian operators on the two $*$-quantum homogeneous spaces of $\asoq{3}$ in \S\ref{sec:*qs}, and gauged versions on bundles over them in the line of \cite{lrz}, in this section we study a
Casimir element. This operator is constructed from the actions of a dual Hopf algebra and is diagonalised in Theorem \ref{cC-diag}.

\subsection{The dual Hopf algebra $\slq$ and its real forms}\label{sec:*structU}

From Drinfel'd--Jimbo construction of quantum universal envelopping algebras it is known that $\cU_{q^{1/2}}(so(3)) \simeq \slq$. 
On the other hand as recalled in \S\ref{sec:dc}, there is an isomorphism $\asoq{3} \simeq \cO(SL_{q^{1/2}})(2) / \IZ_2$. We shall then work out 
a dual pairing between $\asoq{3}$ and $\slqh$. 

The algebra $\slqh$ is generated  by elements $K,K^{-1}, E,F$ subject to the relations
\begin{align}\label{slm}
K^{\pm}E = q^{\pm 1} EK^{\pm}  \, , \quad 
K^{\pm}F =q^{\mp 1} FK^{\pm}  \, , \quad
EF-FE= \frac{K-K^{-1}}{q^\half - q^{-\half}} \;
\end{align}
together with $K K^{-1} = K^{-1}K =1$. It is a Hopf algebra with coproduct, counit and antipode given respectively by
$$
\Delta(K^{\pm 1})= K^{\pm 1} \ot K^{\pm 1}
\; , \quad
\Delta(E)= E \ot K + 1 \ot E
\; , \quad 
\Delta(F)= F \ot 1 +K^{-1} \ot F \; ,
$$
$$
\varepsilon(K^{\pm 1})= 1
\; , \quad
\varepsilon(E)= 0
\; , \quad 
\varepsilon(F)= 0 
$$
$$
S(K^{\pm 1})= K^{\mp 1}
\; , \quad
S(E)= - E K^{-1} 
\; , \quad 
S(F)= -KF \; .
$$
See e.g. \cite[\S 3.1]{ks}.

The non zero values of the pairing 
$\langle \cdot , \cdot \rangle:\slqh \times \asoq{3} \to \IC$ on the algebra generators, besides $\langle 1, \u{k}{k} \rangle=1$ for $k=1,2,3$, and
$\langle K^{\pm 1}, 1 \rangle= 1$, are found to be 
\begin{align}\label{pairing}
& \langle K , \u{1}{1} \rangle = q^{-1}
\; , \qquad 
\langle K , \u{2}{2} \rangle =1 
\; , \qquad 
\langle K , \u{3}{3} \rangle = q \nn
\\ 
& 
\langle K^{-1} , \u{1}{1} \rangle = q
\; , \qquad 
\langle K^{- 1} , \u{2}{2} \rangle =1 
\; , \qquad 
\langle K^{-1} , \u{3}{3} \rangle =  q^{-1} \; , \nn
\\ 
& 
\langle E , \u{2}{1} \rangle = \alpha \eta
\; , \qquad 
\langle E , \u{3}{2} \rangle = - \alpha q^\half \eta \; , 
\nn
\\
& 
\langle F , \u{1}{2} \rangle = \alpha^{-1} \eta
\; , \qquad  
\langle F , \u{2}{3} \rangle = - \alpha^{-1} q^{-\half} \eta\; ,
\end{align}
where $\eta := (q^\half + q^{-\half})^\half$ and $\alpha\in \IC \setminus \{0\}$.

The pairing extends to the whole algebras by the rules 
 $\langle fg,a \rangle=\langle f \ot g ,\Delta(a) \rangle= \langle f ,a_{(1)} \rangle    \langle 
g ,a_{(2)} \rangle$ and
$\langle f,ab \rangle=\langle \Delta(f), a \ot b \rangle= \langle f_{(1)} , a\rangle    \langle 
f_{(2)} b \rangle$, for all $f,g \in \slqh$ and $a,b \in \asoq{3}$.
It satisfies 
$\langle 1, a \rangle=\varepsilon(a)$, $\langle f, 1 \rangle=\varepsilon(f)$ and
$\langle S(f), a \rangle=\langle f, S(a) \rangle$ for  each $f \in \slqh$ and $a \in \asoq{3}$.

\begin{rem}
The extra parameter $\alpha$ in \eqref{pairing} can be re-absorbed by the Hopf algebra automorphism 
of $\slqh$, which rescales $E \mapsto \alpha^{-1} E$, $F \mapsto \alpha F$, 
$K \mapsto K$ \cite[Prop. 3.6]{ks}. We hence fix $\alpha = 1$. 
\end{rem}

It follows by standard arguments in Hopf algebra theory that each left (respectively  right) $\asoq{3}$-comodule algebra $\cA$ carries  a right representation $\rhd $ (respectively  left representation $\lhd $)  of the dual algebra $\slqh$. In details, if $\cA$ is a left comodule algebra via $\rho: \cA \to \asoq{3} \ot \cA$, $a \mapsto a_{(-1)} \ot a_{(0)}$, then $\cA$ carries the right action 
$$
\lhd: \cA \ot \slq \to \cA \; , \quad  a \lhd f := \langle f , a_{(-1)}\rangle   a_{(0)} , \quad a \in \cA, f \in \slq .
$$
If $\cA$ is a right comodule algebra via $\delta: \cA \to \cA \ot \asoq{3}$, $a \mapsto a_{(0)} \ot a_{(1)}$, then $\cA$ carries the  left action 
$$\rhd: \slq \ot \cA \to \cA \; , \quad  f \rhd a := a_{(0)}  \langle f , a_{(1)}\rangle , \quad a \in \cA, f \in \slq .
$$
For $\cA=\asoq{3}$ with left and right coactions given by the coproduct, 
the right and left actions of $\slqh$ on generators $\u{j}{k}$ of $\asoq{3}$ read
$$
\u{j}{k} \lhd f = \langle f , \u{j}{m} \rangle  \u{m}{k}  \qquad \text{and} \qquad f \rhd \u{j}{k} = \u{j}{m} \langle f , \u{m}{k} \rangle .
$$

Explicitly,  the right action is 
\begin{align}\label{ractso}
&  \u{1}{k} \lhd K^{\pm 1} = q^{\mp 1} \u{1}{k} \; , \quad 
&& \u{2}{k}\lhd K^{\pm 1} =  \u{2}{k} \; , \quad 
&&   \u{3}{k} \lhd K^{\pm 1} = q^{\pm 1} \u{3}{k} \; ,   \nn
\\
&   \u{1}{k}\lhd E= 0 \; , \quad 
&& \u{2}{k}\lhd E=  \eta \,  \u{1}{k} \; , \quad 
&&  \u{3}{k}\lhd E= -q^\half \eta \,  \u{2}{k} \; , \nn 
\\
&  \u{1}{k}\lhd F= \eta \,  \u{2}{k} \; , \quad 
&&  \u{2}{k}\lhd F= -q^{-\half} \eta \,  \u{3}{k} \; \quad
&& \u{3}{k}\lhd F=0\; ,
\end{align}
and the left action is given by
\begin{align}\label{lactso}
&  K^{\pm 1} \rhd \u{j}{1} = q^{\mp1} \u{j}{1} \; , \quad 
&&  K^{\pm 1} \rhd  \u{j}{2} =  \u{j}{2} \; , \quad 
&&  K^{\pm 1} \rhd   \u{j}{3} = q^{\pm 1} \u{j}{3} \; ,   \nn
\\
&  E \rhd \u{j}{1} = \eta \, \u{j}{2} \; , \quad 
&& E \rhd  \u{j}{2} = - q^\half  \eta \, \u{j}{3} \; , \quad 
&&  E \rhd \u{j}{3}  = 0 \; , \nn 
\\
&  F \rhd \u{j}{1} = 0  \; , \quad 
&& F \rhd \u{j}{2} = \eta \, \u{j}{1}  \; \quad
&&  F \rhd  \u{j}{3}  = - q^{-\half} \eta \,  \u{j}{2}\; .
\end{align}

Since the left coaction of $\asoq{3}$ on itself descends to $B=\asoq{3}^{co\, \cO(SO(2))}$, see \eqref{coact-left}, the right action \eqref{ractso} preserves $B$. Explicitly, on the generators $\y{k}:=\u{k}{2}$ of $B$, the action $\lhd: B \ot \slq \to B$ is given by
\begin{align}
&  \y{1} \lhd K^{\pm 1} = q^{\mp 1} \y{1} \; , \quad 
&& \y{2}\lhd K^{\pm 1} =  \y{2} \; , \quad 
&&   \y{3} \lhd K^{\pm 1} = q^{\pm 1} \y{3} \; ,   \nn
\\
&   \y{1}\lhd E= 0 \; , \quad 
&& \y{2}\lhd E=  \eta \,  \y{1} \; , \quad 
&&  \y{3}\lhd E= -q^\half \eta \,  \y{2} \; , \nn 
\\
&  \y{1}\lhd F= \eta \,  \y{2} \; , \quad 
&&  \y{2}\lhd F= -q^{-\half} \eta \,  \y{3} \; \quad
&& \y{3}\lhd F=0\; .
\end{align}
For the left action \eqref{lactso} this is not the case. The generators $E$ and $F$ do not preserve $B$ while the generator $K$ does and acts as the identity. Its left action is indeed dual to the right coaction in \eqref{coact-right} of the generator $z$ of $\cO(SO(2))$ on $\asoq{3}$ and we could equivalently define the algebra of coinvariant elements $B$ as given by invariants
\beq
B = \{ b \in \asoq{3} \, | \, K \rhd b = b \} . 
\eeq

Depending on the values of the deformation parameter $q$, the Hopf algebra $\slqh$ can be equipped with the following real structures \cite[\S 3.1.4]{ks}:
\begin{enumerate}[$\bullet$]
\item 
if $q \in \IR$, there are two (non equivalent)  $*$-structures:
\beq
(K^{\pm 1})^*=K^{\pm 1} \; , \quad E^*= FK \; , \quad F^*= K^{-1} E  
\eeq
with corresponding Hopf $*$-algebra $\uqh(su_2)$ (this is the compact real form) and
\beq
(K^{\pm 1})^*=K^{\pm 1} \; , \quad E^*= -FK \; , \quad F^*= -K^{-1} E
\eeq
with corresponding Hopf $*$-algebra $\uqh(su_{1,1})$; \\

\item
if $|q|=1$ there is only one $*$-structure given by
\beq
(K^{\pm 1})^*=K^{\pm 1} \; , \quad E^*= -E \; , \quad F^*= -F  \;.
\eeq
The corresponding Hopf $*$-algebra is $\uqh(sl_2(\IR))$. 
Classically the Lie algebras $su_{1,1}$ and $sl_2(\IR)$ are isomorphic. 
\end{enumerate}
 
The pairing \eqref{pairing} induces a pairing between the real forms $\uqh(su_2)$ and $\cO(SO_q(3;\IR))$ and between the real forms $\uqh(sl_2(\IR))$ and 
$\cO(SO_q(1,2))$. Indeed the conditions 
\beq\label{parsc}
\langle f^*, a \rangle=\overline{\langle f, S(a)^* \rangle}
\; , \quad
\langle f, a^* \rangle=\overline{\langle S(f)^*,a \rangle} 
\eeq
are satisfied for each $f \in \slqh$ and $a \in \cO(SO_q(3;\IR))$ or $f \in \uqh(sl_2(\IR))$ and $a \in \cO(SO_q(1,2))$. On the other hand, the condition \eqref{parsc} for the pairing \eqref{pairing}  is not satisfied for the algebra $\uqh(su_{1,1})$.

\medskip
We need some notation. For $n \in \IN$ the $q$-integer is defined as  
\beq\label{q-no}
[n] := [n]_{q^\half} := \frac{q^{\frac{n}{2}} - q^{-\frac{n}{2}}}{q^{\half} - q^{-\half}}. 
\eeq
It has properties
\begin{align}
[n] = q^{\frac{n}{2} -\half} \sum_{j=0}^{n-1} q^{-j}  = q^{\frac{-n+1}{2}} \sum_{j=0}^{n-1} q^{j}, \qquad 
[n] = [2] [n-1] - [n-2] .
\end{align}

When the deformation parameter $q$ is not a root of unity, the centre of the algebra $\slqh$ is generated by 
the (quadratic) Casimir element (see \cite[\S 3.1.1]{ks}):
\begin{align}\label{cas}
C_q &:=EF + \frac{q^{-\half}K + q^\half K^{-1}}{(q^\half - q^{-\half})^2} 
= FE + \frac{q^\half K + q^{-\half}  K^{-1}}{(q^\half - q^{-\half})^2}  \nn \\
&\:= \half (EF+FE) + \frac{ q^\half + q^{-\half} }{(q^\half - q^{-\half})^2}  (K + K^{-1}) .
\end{align}

We would like to diagonalise the Casimir as an operator acting on the left on $B$ and use the right action of $\slqh$
to construct a basis of eigenfunctions, since clearly $C_q \rhd (a \lhd f) = (C_q \rhd a) \lhd f$.
As mentioned, 
while $E$ and $F$ do not preserve $B$, both the products $EF$ and $FE$ do. 
On the other hand, the generators $K, K^{-1}$ act on $B$ as the identity and hence
$$
\frac{q^{-\half}K + q^\half K^{-1}}{(q^\half - q^{-\half})^2} \, \rhd b  =\frac{q^{-\half} + q^\half }{(q^\half - q^{-\half})^2} \, b , \quad b \in B .
$$
Thus, we can remove from the Casimir an additive constant and consider the operator 
\beq\label{casb}
\cC_q := C_q - \frac{q^{-\half} + q^\half }{(q^\half - q^{-\half})^2} = EF = FE
\eeq
acting on the left on $B$. On the generators, the  action of $\cC_q$ is easily found to be 
\beq\label{cqn1}
\cC_q \rhd y_k = \eta^2 y_k = [2] y_k , \quad k=1,2,3. 
\eeq

\begin{prop}
There is a vector space decomposition 
$$
B = \oplus_{J\in\IN} V_J 
$$
into irreducible representations $V_J$ of $\slqh$. The spaces $V_J$ are given by
\beq\label{vj}
V_J = \textup{span} \{\y{3}^J \lhd E^m\} = \textup{span}\{y_1^J \lhd F^m\}, \quad
m=0,1, \dots, 2J . 
\eeq
Thus $\y{3}^J$ (respectively $y_1^J$) is the highest (respectively lowest) weight vector of the representation.
\end{prop}
\begin{proof}
The proof is analogous to the one in \cite[\S 4.5.2]{ks}.
\end{proof}

\begin{thm} \label{cC-diag}
For each $J\in\IN$ the elements in $V_J$ are eigenfunctions of $\cC_q$ with eigenvalue 
$[J][J+1]$:
\beq 
\cC_q \rhd a = [J][J+1] \, a , \quad \forall a \in V_J. 
\eeq
\end{thm}
\begin{proof}
In view of \eqref{vj} it is enough to show \eqref{cC-diag} for the highest weight vector 
$\y{3}^J$. Clearly, if $\cC_q \rhd \y{3}^J = [J][J+1] \y{3}^J$, then for each $m=0,1, \dots, 2J$,
$$
\cC_q \rhd (\y{3}^J \lhd E^m) = (\cC_q \rhd \y{3}^J) \lhd E^m =  
[J][J+1] (\y{3}^J \lhd E^m) \,.
$$
Indeed we can show the result at once for the lowest and highest weight vectors.
Using the coproduct 
$$
\Delta(EF)= EF \ot K + K^{-1} \ot EF + q^{-1} EK^{-1} \ot FK + F \ot E 
$$
and recalling from \eqref{lactso} that $K$ and $K^{-1}$ act as the identity on the elements of $B$, the operator $\cC_q$ acts on the product of two elements $a,a'$ as
\begin{align}\label{Dcq}
\cC_q \rhd(aa')= & ((EF) \rhd a) a' + a ((EF) \rhd a') + q^{-1} (E \rhd a) (F \rhd a') + (F\rhd a)(E\rhd a')  \nn
\\
= & (\cC_q \rhd a) a' + a (\cC_q\rhd a') + q^{-1} (E \rhd a) (F \rhd a') + (F\rhd a)(E\rhd a') \, .
\end{align}
We hence need to compute the action of $E$ and $F$ on any power $\y{\ell}^J$ of $\y{\ell}$, $\ell=1,3$. By induction on $n$ one shows that
\begin{align*}
E \rhd \y{\ell}^n &= -q^\half \eta (\sum_{j=0}^{n-1} q^{-j}) \y{\ell}^{n-1} \u{\ell}{3}
= -q^{-\frac{n}{2}+1} \eta [n] \y{\ell}^{n-1} \u{\ell}{3}
\\
F \rhd \y{\ell}^n &=  \eta (\sum_{j=0}^{n-1} q^{j}) \y{\ell}^{n-1} \u{\ell}{1}
= q^{\frac{n-1}{2}} \eta [n] \y{\ell}^{n-1} \u{\ell}{1} \; ,
\end{align*}
where $[n]$ is the $q^\half$-number in \eqref{q-no}.
Next, we prove that $\cC_q \rhd \y{\ell}^n = [n][n+1] \y{\ell}^n$  by induction on $n$. The result holds for  the base case $n=1$, as already observed in \eqref{cqn1}. Assume it holds for $n$, then, by also using \eqref{Dcq}, we compute
\begin{align*}
\cC_q \rhd (\y{\ell}^{n+1}) 
& =  
(\cC_q \rhd \y{\ell}^n)  \y{\ell} + \y{\ell}^n (\cC_q\rhd  \y{\ell}) + q^{-1} (E \rhd \y{\ell}^n) (F \rhd  \y{\ell}) + (F\rhd \y{\ell}^n)(E\rhd  \y{\ell}) 
\\
& = [n][n+1] \y{\ell}^{n+1}+ [2] \y{\ell}^{n+1} -q^{-\frac{n}{2}} \eta^2 [n] \y{\ell}^{n-1} \u{\ell}{3} \u{\ell}{1} - q^{\frac{n}{2}} \eta^2 [n] \y{\ell}^{n-1} \u{\ell}{1}\u{\ell}{3} 
\end{align*}
where $\eta^2=[2]= q^{-\half} (1+q)$ and 
$$
\u{\ell}{1}\u{\ell}{3} = -q^{\frac{3}{2}} (1+q)^{-1} \y{\ell}^2  \, , \quad
\u{\ell}{3} \u{\ell}{1}  = -q^{-\half} (1+q)^{-1} \y{\ell}^2 \; ,
$$
as from the expressions in Appendix \ref{app:B}. 
We hence obtain that $\y{\ell}^{n+1}$ is an eigenfunction of $\cC_q$
with eigenvalue
\begin{multline*}
[n][n+1] + [2] 
+q^{-\frac{n+1}{2}} [2] [n]  (1+q)^{-1} 
+ q^{\frac{n+3}{2}} [2] [n] (1+q)^{-1} =
\\
= [n][n+1] + [2] 
+q^{-\frac{n+2}{2}}  [n] 
+ q^{\frac{n+2}{2}}  [n] \; .
\end{multline*}
Next, by explicit computation one verifies that 
$$
 [2] +q^{-\frac{n+2}{2}}  [n] + q^{\frac{n+2}{2}}  [n] 
=
[n+1] ([2][n+1]-2 [n])
$$
so that, finally,
$$
\cC_q \rhd (\y{\ell}^{n+1})= 
[n+1] \big([n]+[2][n+1]- 2[n]\big)  \y{\ell}^{n+1}
= [n+1][n+2]\,  \y{\ell}^{n+1}
$$
where we have used the property  $[2][n+1]- [n]=[n+2]$ of $q$-numbers.
\end{proof}

The above analysis is valid when $q$ is real and for the dual $*$-algebras $\uqh(su(2))$ and $\asph$. 
The  more complicate case $|q|=1$ that involves unbounded representations of $\uqh(sl_{2}(\IR))$ \cite{Sc96} will be studied elsewhere.

\appendix

\section{Proof of Proposition \ref{pro:cof} }\label{app:A}

From the definition \eqref{cofactors}, we are left to show that $\sum_m \u{d}{m} \hu{m}{a}=0$,
for all indices $a\neq d$.
Notice that for each index $a=1,2,3$ (and for each $m$) we can always choose an expression of the cofactor
$
\hu{m}{a}= \varepsilon_{abc}^{-1} \sum_{n,p} \varepsilon_{mnp}  \u{b}{n}\u{c}{p}  $
for which $a,b,c$ are all different. So either $d=b$ or $d=c$.  
Without loss of generality we can take $d=b$ (that is, of the two equivalent expressions of the cofactor with $a\neq b \neq c$ we can take the one where the index $b$ is equal to $d$). Thus,  fixing mutually different indices $a, b=d, c$, we compute
\begin{align}\label{da}
\varepsilon_{adc} \, \sum_m \u{d}{m} \hu{m}{a} &=
\sum_{m,n,p} \varepsilon_{mnp}  \u{d}{m} \u{d}{n}\u{c}{p}  
\nn \\
&= \sum_{m,n} \varepsilon_{mn1}  \u{d}{m} \u{d}{n}\u{c}{1}   
+ \sum_{m,n} \varepsilon_{mn2}  \u{d}{m} \u{d}{n}\u{c}{2}
+ \sum_{m,n} \varepsilon_{mn3}  \u{d}{m} \u{d}{n}\u{c}{3}  
\nn
\\
&= q(\u{d}{2} \u{d}{3}-q\u{d}{3} \u{d}{2})\u{c}{1}   
- q(\u{d}{1} \u{d}{3}-\u{d}{3} \u{d}{1}+  (q^\half -q^{-\half}) \u{d}{2}\u{d}{2}   )\u{c}{2}    \nn
\\
& \quad + (\u{d}{1} \u{d}{2}-q\u{d}{2} \u{d}{1})\u{c}{3} \; . 
\end{align}
We then use equation \eqref{rtt-exp} for elements $\u{d}{m}$ on the same row:
\begin{align} \label{da0}
q^{1-\delta_{d2}} \u{d}{m}  \u{d}{n}  = &~  q^{\delta_{mn}-\delta_{mn'}}  \u{d}{n} \u{d}{m} 
+ \lambda \theta(n-m) \u{d}{m}  \u{d}{n} + 
\delta_{d2} \lambda   q^{-\half}\u{1}{m}  \u{3}{n} \nn \\ &  
- \lambda \delta_{nm'} \sum_k \theta(k-m)  q^{-\rho_m-\rho_{k'}}\u{d}{k'}  \u{d}{k} \, . 
\end{align}
For $d \neq 2$, this yields
\begin{align*}
q \u{d}{3}  \u{d}{2}  & =  \u{d}{2} \u{d}{3} \; , \quad
q \u{d}{2}  \u{d}{1}  =  \u{d}{1} \u{d}{2} \; ,  \\
q^2 \u{d}{3}  \u{d}{1}  & =  \u{d}{1} \u{d}{3} \; , \quad
(1+q^{-1}) \u{d}{1} \u{d}{3} = q^{-2} \u{d}{1} \u{d}{3} + q^{-1} \u{d}{3} \u{d}{1} - q^{-\half} \lambda \u{d}{2} \u{d}{2} . 
\end{align*}
The first two relations imply the vanishing of the (polynomial) coefficients of  $\u{c}{1} $ and $\u{c}{3}$. The last two when combined yield
$$
(1+q^{-1}) \u{d}{1} \u{d}{3} = (1+q^{-1})\u{d}{3} \u{d}{1} - (1+q^{-1})(q^\half -q^{-\half}) \u{d}{2}\u{d}{2}
$$
and 
the coefficient of  $\u{c}{2} $ vanishes as well.

\medskip

For $d=2$ the computation is more involved. Equation \eqref{da} becomes 
\begin{align}\label{da2}
\varepsilon_{a2c} \, \sum_m \u{2}{m} \hu{m}{a} & = q(\u{2}{2} \u{2}{3}-q\u{2}{3} \u{2}{2})\u{c}{1}   
- q(\u{2}{1} \u{2}{3}-\u{2}{3} \u{2}{1}+  (q^\half -q^{-\half}) \u{2}{2}\u{2}{2}   )\u{c}{2}    \nn
\\
& \quad + (\u{2}{1} \u{2}{2}-q\u{2}{2} \u{2}{1})\u{c}{3} \; ,
\end{align}
with the coefficients of the $\u{c}{p}$ that do not vanish, in contrast to the case $d=1,3$. We hence need to proceed differently:
the idea is to express the coefficients as polynomials in $\u{3}{k}\u{1}{j}$
for the case $c=1$ or as polynomials in $\u{1}{k}\u{3}{j}$ for the case $c=3$. 
We start with the coefficient of $\u{c}{1}$. The equation \eqref{da0} yields
\begin{align*}
\u{2}{3} \u{2}{2} &= \u{2}{2} \u{2}{3} +  q^{-\half} \lambda \u{1}{3} \u{3}{2} \\
(1+q^{-1}) \u{2}{2} \u{2}{3} & = \u{2}{3} \u{2}{2} + q \u{2}{2} \u{2}{3} + q^{-\half} \lambda \u{1}{2} \u{3}{3} .
\end{align*}
When combined, these yield
$$
\u{2}{2} \u{2}{3}-q\u{2}{3} \u{2}{2} = (q^\half -q^{-\half}) (\u{1}{2} \u{3}{3} - q \u{1}{3} \u{3}{2}) . 
$$
This can also be written as 
$$
\u{2}{2} \u{2}{3}-q\u{2}{3} \u{2}{2} = (q^\half -q^{-\half}) ( q \u{3}{3} \u{1}{2} - \u{3}{2} \u{1}{3}) 
$$
when using the commutation relations
$$
q^{-1} \u{1}{3}  \u{3}{2} = \u{3}{2}  \u{1}{3}  \; , \qquad 
q^{-1} \u{1}{2}  \u{3}{3} = \u{3}{3}  \u{1}{2} + \lambda \u{3}{2}  \u{1}{3}
$$
obtained from \eqref{rtt-exp}, for suitable choices of indices. 

Analogously, for the coefficient of $\u{c}{3}$, from equation \eqref{da0} we obtain
\begin{align*}
\u{2}{2} \u{2}{1} &= \u{2}{1} \u{2}{2} +  q^{-\half} \lambda \u{1}{2} \u{3}{1} \\
(1+q^{-1}) \u{2}{1} \u{2}{2} & = \u{2}{2} \u{2}{1} + q \u{2}{1} \u{2}{2} + q^{-\half} \lambda \u{1}{1} \u{3}{2} .
\end{align*}
When combined, these yield
$$
\u{2}{1} \u{2}{2}-q\u{2}{2} \u{2}{1} = (q^\half -q^{-\half}) (\u{1}{1} \u{3}{2} - q \u{1}{2} \u{3}{1}) . 
$$
This can also be written as 
$$
\u{2}{1} \u{2}{2}-q\u{2}{2} \u{2}{1} = (q^\half -q^{-\half}) ( q \u{3}{2} \u{1}{1} - \u{3}{1} \u{1}{2}) 
$$
when using the commutation relations
$$
q^{-1} \u{1}{2}  \u{3}{1} = \u{3}{1}  \u{1}{2}  \; , \qquad 
q^{-1} \u{1}{1}  \u{3}{2} = \u{3}{2}  \u{1}{1} + \lambda \u{3}{1}  \u{1}{2} .
$$
again obtained from \eqref{rtt-exp}, for suitable choices of indices.   

Finally, the coefficient of $\u{c}{2}$ in \eqref{da2} is proportional to the cofactor $\hu{2}{2}$:
\begin{align*}
\u{2}{1}\u{2}{3}-  \u{2}{3}\u{2}{1} & + (q^\half -q^{-\half}) \u{2}{2}\u{2}{2} = 
(q^\half -q^{-\half}) \hu{2}{2} \\
& =  (q^\half -q^{-\half}) \left[ \u{1}{1}\u{3}{3} - \u{1}{3}\u{3}{1}  + (q^\half -q^{-\half}) \u{1}{2}\u{3}{2} \right] \\
& = (q^\half -q^{-\half}) \left[ - \u{3}{1}\u{1}{3} + \u{3}{3}\u{1}{1}  - (q^\half -q^{-\half}) \u{3}{2}\u{1}{2} \right].
\end{align*}
We  then return to \eqref{da2}. For $c=1$ equation \eqref{da2}  reads
\begin{align*}
 -q^{2}\, \sum_m \u{2}{m} \hu{m}{3} 
 & = q (\u{2}{2} \u{2}{3}-q\u{2}{3} \u{2}{2})\u{1}{1} + (\u{2}{1} \u{2}{2}-q\u{2}{2} \u{2}{1})\u{1}{3}  \\
 & \quad - q \big( \u{2}{1} \u{2}{3}-\u{2}{3} \u{2}{1} + (q^\half -q^{-\half}) \u{2}{2}\u{2}{2} \big) \u{1}{2}    
 \\ 
 & = (q^\half -q^{-\half}) \big[ q ( q \u{3}{3} \u{1}{2} - \u{3}{2} \u{1}{3} ) \u{1}{1} + ( q \u{3}{2} \u{1}{1} - \u{3}{1} \u{1}{2}) \u{1}{3} 
\\ & \quad -q \big( - \u{3}{1}\u{1}{3} + \u{3}{3}\u{1}{1}  - (q^\half -q^{-\half}) \u{3}{2}\u{1}{2} \big)\u{1}{2} \big] \\
& = q (q^\half -q^{-\half}) \, \u{3}{2} \big[ - \u{1}{3} \u{1}{1} + \u{1}{1} \u{1}{3} + 
(q^\half -q^{-\half}) \u{1}{2}\u{1}{2} \big] ,
\end{align*}
where in the last equality we have used 
$$
\u{1}{2}\u{1}{1}  =q^{-1} \u{1}{1}\u{1}{2} \quad \text{and} \quad
\u{1}{3}\u{1}{2}  =q^{-1} \u{1}{2}\u{1}{3} \; , \quad
$$
obtained once again from \eqref{rtt-exp}. From \eqref{rtt-exp} we also obtain
$$
\u{1}{3}\u{1}{1}  =q^{-2} \u{1}{1}\u{1}{3} \quad \text{and} \quad 
(1+ q^{-1}) \u{1}{1} \u{1}{3} = q^{-1} \u{1}{3} \u{1}{1} + q^{-2} \u{1}{1} \u{1}{3} - q^{-\half} \lambda \u{1}{2} \u{1}{2} 
$$
which, when combined, give
$$
\u{1}{1} \u{1}{3} = \u{1}{3} \u{1}{1} - (q^\half -q^{-\half}) \u{1}{2}\u{1}{2} 
$$
and then $\sum_m \u{2}{m} \hu{m}{3} = 0$.   

Similarly, for $c=3$ equation \eqref{da2} reads
\begin{align*}
 \sum_m \u{2}{m} \hu{m}{1} & = q(\u{2}{2} \u{2}{3}-q\u{2}{3} \u{2}{2})\u{3}{1} + (\u{2}{1} \u{2}{2}-q\u{2}{2} \u{2}{1})\u{3}{3} 
\\ & - q \big(\u{2}{1} \u{2}{3}-\u{2}{3} \u{2}{1} + (q^\half -q^{-\half}) \u{2}{2}\u{2}{2} \big) \u{3}{2}   
\\ & = 
(q^\half -q^{-\half}) \big[ q ( \u{1}{2} \u{3}{3} - q \u{1}{3} \u{3}{2} ) \u{3}{1} + ( \u{1}{1} \u{3}{2} - q \u{1}{2} \u{3}{1}) \u{3}{3} 
\\ & \quad - q \big( \u{1}{1}\u{3}{3} - \u{1}{3}\u{3}{1} + (q^\half -q^{-\half}) \u{1}{2}\u{3}{2} \big) \u{3}{2} \\
& = q (q^\half -q^{-\half}) \, \u{1}{2} \big[ \u{3}{3} \u{3}{1} - \u{3}{1} \u{3}{3} - (q^\half -q^{-\half}) \u{3}{2}\u{3}{2} \big] ,
\end{align*}
where in the last equality we have used 
$$
\u{3}{2}\u{3}{1}  =q^{-1} \u{3}{1}\u{3}{2} \; , \quad
\u{3}{3}\u{3}{2}  =q^{-1} \u{3}{2}\u{3}{3} \; , \quad
$$
obtained once again from \eqref{rtt-exp}. From \eqref{rtt-exp} we also obtain
$$
\u{3}{3}\u{3}{1}  =q^{-2} \u{3}{1}\u{3}{3} \quad \text{and} \quad 
(1+ q^{-1}) \u{3}{1} \u{3}{3} = q^{-1} \u{3}{3} \u{3}{1} + q^{-2} \u{3}{1} \u{3}{3} - q^{-\half} \lambda \u{3}{2} \u{3}{2} 
$$
which, when combined, give
$$
\u{3}{1} \u{3}{3} = \u{3}{3} \u{3}{1} - (q^\half -q^{-\half}) \u{3}{2}\u{3}{2} 
$$
and then $\sum_m \u{2}{m} \hu{m}{1} = 0$.  
This concludes the proof of Prop. \ref{pro:cof}.

\section{Commutation relations in $\oq{3}$}\label{app:cr}

In this appendix, we compute explicitly the commutation relations \eqref{rtt-exp} 
among the generators $\u{i}{j}$ of the algebra 
$\oq{3}$, for $j=1,3$,  which we need for computing the coinvariant elements in Proposition \ref{prop:B}.

As before
$\lambda= q - q^{-1}$, and $\rho_1=\half$, $\rho_2=0$, $\rho_3=-\half$.  Moreover, for each index $k=1,2,3$, $k'=3-k$ so that  $1'=3$, $2'=2$ and $3'=1$.

\subsubsection*{Commutation relations $\u{i}{1}$ $\u{j}{1}$}
For $m=n=1$,  equation \eqref{rtt-exp} reduces to
$$
q^{\delta_{ij}-\delta_{ij'}} \u{j}{1}  \u{i}{1}  = ~  (q - \lambda \theta(j-i) ) \u{i}{1} \u{j}{1} 
  + \lambda \delta_{ij'} \sum_k \theta(j-k)  q^{-\rho_i-\rho_k}\u{k}{1}  \u{k'}{1}  \, , 
$$
from which  
\beq
\begin{array}{lll}
\u{2}{1}  \u{1}{1}  =   q^{-1}  \u{1}{1} \u{2}{1} \, , \qquad
&
\u{3}{1}  \u{1}{1}  =  q^{-2} \u{1}{1} \u{3}{1}    \, , 
\\
 \u{3}{1}  \u{2}{1}  = ~  q^{-1} \u{2}{1} \u{3}{1} \, , \qquad
&
(\u{2}{1})^2  =  -q^{-\frac{3}{2}} (1+ q)\u{1}{1}  \u{3}{1}   \, .
\end{array}
\eeq

\subsubsection*{Commutation relations $\u{i}{3}$ $\u{j}{3}$}
For $m=n=3$,  equation \eqref{rtt-exp} has an  expression analogous to that for $m=n=1$:
$$
q^{\delta_{ij}-\delta_{ij'}} \u{j}{3}  \u{i}{3}  = ( q - \lambda  \theta(j-i) ) \u{i}{3} \u{j}{3} 
+ \lambda \delta_{ij'} \sum_k \theta(j-k)  q^{-\rho_i-\rho_k}\u{k}{3}  \u{k'}{3}   
$$
and one has 
\beq
\begin{array}{lll}
\u{2}{3}  \u{1}{3}  =   q^{-1}  \u{1}{3} \u{2}{3} \, , \qquad
\u{3}{3}  \u{1}{3}  =  q^{-2} \u{1}{3} \u{3}{3}    \, , 
\\
 \u{3}{3}  \u{2}{3}  = ~  q^{-1} \u{2}{3} \u{3}{3} ~ \, , \qquad
(\u{2}{3})^2  =  -q^{-\frac{3}{2}} (1+ q)\u{1}{3}  \u{3}{3}   \, , 
\end{array}
\eeq

\subsubsection*{Commutation relations $\u{i}{1}$ $\u{j}{3}$}
For $m=3$ and $n=1$,  equation \eqref{rtt-exp} gives
$$
q^{-1}  \u{i}{1} \u{j}{3} =
q^{\delta_{i j}-\delta_{i j'}} \u{j}{3}  \u{i}{1}  
+ \lambda\theta(j-i)\u{i}{3}  \u{j}{1} 
- \lambda \delta_{i j'} \sum_k \theta(j-k)  q^{-\rho_i-\rho_k}\u{k}{3}  \u{k'}{1}   
$$
from which
\begin{align}\label{com-rel13}
& 
\u{1}{3}  \u{1}{1} =q^{-2}  \u{1}{1} \u{1}{3} \, ,  \qquad 
\u{2}{1} \u{1}{3} =q ~ \u{1}{3}  \u{2}{1}  \, ,  \qquad
 \u{2}{3}  \u{1}{1}  =  q^{-1} \u{1}{1} \u{2}{3}  -  \lambda \u{1}{3}  \u{2}{1} \, ,  \nn \\
&
\u{2}{3}  \u{2}{1} = q^{-1} \u{2}{1} \u{2}{3} + q^{-\half} \lambda  \u{1}{3}  \u{3}{1} \, , \qquad
\u{3}{1} \u{2}{3} = q~\u{2}{3}  \u{3}{1}\, ,   \qquad 
\u{3}{1} \u{1}{3} =\u{1}{3}  \u{3}{1}  \, , \nn \\
&  
 \u{3}{3}  \u{1}{1}  = \u{1}{1} \u{3}{3} + (1-q^{-1}) \lambda  \u{1}{3}  \u{3}{1} 
+ \lambda   q^{-\half}  \u{2}{1}  \u{2}{3}  \, , \nn \\
&
\u{3}{3}  \u{2}{1}  = q^{-1} \u{2}{1} \u{3}{3}  - \lambda \u{2}{3}  \u{3}{1} \, , \qquad
 \u{3}{3}  \u{3}{1} =q^{-2} \u{3}{1} \u{3}{3} 
\end{align}

The quotient algebra of $\mathcal{O}(R)$ by the ideal generated by $Q_q-1$
 gives the algebra $\oq{3}$, where, as from\eqref{condq}, $Q_q$ can equivalently be expressed in terms of any index $j$ as 
$$
Q_q= \sum_k q^{\rho_j-\rho_k}  \u{k}{j}  \u{k'}{j'} = \sum_k q^{\rho_j-\rho_k} \u{j}{k} \u{j'}{k'} .
$$ 
Explicitly
\begin{align*}
Q_q &=    \u{1}{1}  \u{3}{3} +
q^{\half}  \u{2}{1}  \u{2}{3}
+ q\u{3}{1}  \u{1}{3}
=  \u{1}{1} \u{3}{3} 
+q^{\half} \u{1}{2} \u{3}{2} 
+ q \u{1}{3} \u{3}{1} 
\\
&=  q^{-\half}  \u{1}{2}  \u{3}{2} +
 \u{2}{2}  \u{2}{2}
+q^{\half}  \u{3}{2}  \u{1}{2}
= q^{-\half} \u{2}{1} \u{2}{3} 
+ \u{2}{2} \u{2}{2} 
+q^{\half} \u{2}{3} \u{2}{1} 
\\
&=  q^{-1}  \u{1}{3}  \u{3}{1} +
q^{-\half}  \u{2}{3}  \u{2}{1}
+  \u{3}{3}  \u{1}{1}
= q^{-1} \u{3}{1} \u{1}{3} 
+q^{-\half} \u{3}{2} \u{1}{2} 
+ \u{3}{3} \u{1}{1} \; ,
\end{align*}
the diagonal entries of the matrices $S(u)u$ and $uS(u)$.

\section{Cofactors and coinvariant elements}\label{app:B}

We  list all the cofactors of the elements of the defining matrix $u$: 
\begin{align*}
\hu{1}{1}&=  \u{2}{2}\u{3}{3}  - q\u{2}{3}\u{3}{2} =
-q^{-1}  \u{3}{2}\u{2}{3}+ \u{3}{3}  \u{2}{2}
\\
\hu{2}{1}&= -q  \u{2}{1}\u{3}{3} + q \u{2}{3} \u{3}{1}  -q   (q^\half -q^{-\half}) \u{2}{2}\u{3}{2}
\\
&=
 \u{3}{1}\u{2}{3} - \u{3}{3} \u{2}{1}  +   (q^\half -q^{-\half}) \u{3}{2}\u{2}{2} 
\\
\hu{3}{1}&= q \u{2}{1}\u{3}{2}  -q^2 \u{2}{2}\u{3}{1}=
- \u{3}{1}\u{2}{2} +q \u{3}{2}\u{2}{1}
\end{align*}
together with
\begin{align*}
\hu{1}{2}&= -q^{-1} \u{1}{2}\u{3}{3} +\u{1}{3}\u{3}{2} =
q^{-1} \u{3}{2}\u{1}{3} - \u{3}{3}  \u{1}{2}
\\
&=
-q^{-1} (q^\half -q^{-\half})^{-1} (\u{2}{2}\u{2}{3}- q \u{2}{3}\u{2}{2} )
\\
\hu{2}{2}&=  \u{1}{1}\u{3}{3}-  \u{1}{3}\u{3}{1}  + (q^\half -q^{-\half}) \u{1}{2}\u{3}{2} =
- \u{3}{1}\u{1}{3}+  \u{3}{3}\u{1}{1}  - (q^\half -q^{-\half}) \u{3}{2}\u{1}{2} 
\\
&=
(q^\half -q^{-\half})^{-1} 
(\u{2}{1}\u{2}{3}-  \u{2}{3}\u{2}{1} + (q^\half -q^{-\half}) \u{2}{2}\u{2}{2} )
\\
\hu{3}{2}&= -\u{1}{1}\u{3}{2} +q \u{1}{2}\u{3}{1} =
q^{-1} \u{3}{1}\u{1}{2} - q \u{3}{2}  \u{1}{1}
\\
&=
(q^\half -q^{-\half})^{-1} (- \u{2}{1}\u{2}{2}+ q \u{2}{2}\u{2}{1} )
\end{align*}
and finally
\begin{align*}
\hu{1}{3}&=  q^{-1} \u{1}{2}\u{2}{3}  - \u{1}{3}\u{2}{2} =
-q^{-2}  \u{2}{2}\u{1}{3}+q^{-1} \u{2}{3}  \u{1}{2}
\\
\hu{2}{3}&= - \u{1}{1}\u{2}{3} + \u{1}{3} \u{2}{1}  -  (q^\half -q^{-\half}) \u{1}{2}\u{2}{2}
\\
&=
q^{-1} \u{2}{1}\u{1}{3} - q^{-1}\u{2}{3} \u{1}{1}  +  q^{-1} (q^\half -q^{-\half}) \u{2}{2}\u{1}{2} 
\\
\hu{3}{3}&= \u{1}{1}\u{2}{2}  -q \u{1}{2}\u{2}{1}=
-q^{-2} \u{2}{1}\u{1}{2} +q^{-1} \u{2}{2}\u{1}{1} \, .
\end{align*}

\medskip
Next, we  list all quadratic coinvariant elements $\u{i}{3}\u{j}{1}$ and $\u{i}{1}\u{j}{3}$  as polynomials in the elements of the second column $\u{k}{2}=:\y{k}$. 
From the proof of Proposition \ref{prop:B} we have 
\begin{align*}
 & \u{1}{3}   \u{1}{1} = -q^{-\half} (1+q)^{-1} \y{1}^2 \, , 
\qquad
\u{1}{3} \u{2}{1}=  q^{-\half} (1+q)^{-1} ~ \y{1} \left(1-\y{2} \right) \, , 
\\
& \u{1}{3} \u{3}{1}=  (1+q)^{-1}(1-\y{2} - q^{-\half} \y{1}\y{3}) ,
\qquad  \u{2}{3} \u{1}{1} = - q^{\half} (1+q)^{-1}  (1 + q^{-1} \y{2}) \y{1} \, , 
\\
& \u{2}{3} \u{2}{1} =\y{3} \y{1} \, ,
\\
& \u{2}{3}  \u{3}{1} = q^{-\half} (1+q)^{-1}  \left( 1 -  \y{2} \right) \y{3} \, , 
\qquad \u{3}{3}\u{1}{1} = (1+q)^{-1} (q+\y{2} - q^{-\half}\y{3}\y{1})
\\
& \u{3}{3} \u{2}{1}=  - q^{-\half} (1+q)^{-1} ~\y{3} \left(q + \y{2}\right) \, , 
\qquad \u{3}{3}   \u{3}{1} = -q^{-\half} (1+q)^{-1} \y{3}^2 \; .
\end{align*}

Formulas for the elements  $\u{i}{1}\u{j}{3}$ are recovered by using \eqref{cr13}, or explicitly \eqref{com-rel13}, and 
also the commutation relations  \eqref{comm-rel-y} 
\begin{eqnarray*}
&&\y{3} (\y{2}-1) = q^{-1} (\y{2}-1) \y{3}   \, , 
\qquad \y{1} (\y{2}-1) = q (\y{2}-1)  \y{1} \, , 
\\
&& q \y{3} \y{1} = q^{-1} \y{1} \y{3} + (q^{-\frac{1}{2}}-q^{\frac{1}{2}}) (\y{2}-1)
\end{eqnarray*}
or equivalently
\begin{eqnarray*}
&&\y{3} \y{2} = q^{-1} \y{2} \y{3} +(1-q^{-1}) \y{3}  \, , 
\qquad \y{2} \y{1}= q^{-1} \y{1} \y{2} +(1-q^{-1}) \y{1}  \, ,
\\
&&\y{3} \y{1} = q^{-2} \y{1} \y{3} + q^{-\frac{3}{2}}(1-q) (\y{2}-1)
\end{eqnarray*}
with \eqref{quadric-y}: $q^{-\half}\y{1} \y{3} +q^{\half} \y{3} \y{1} +\y{2}^2=1$.
Finally for the remaining coinvariant elements
\begin{align*}
& \u{1}{1}   \u{1}{3} = -q^{\frac{3}{2}} (1+q)^{-1} \y{1}^2 \, , 
\qquad \u{1}{1} \u{2}{3}= - q^\half (1+q)^{-1} \y{1} \left( 1+q~ \y{2} \right) \, ,
\\
 &\u{1}{1} \u{3}{3} = (1+q)^{-1}  (1+ q \y{2}- q^{\frac{3}{2}} \y{1}\y{3})
\end{align*}
\begin{align*}
&\u{2}{1}  \u{1}{3} = ~ q^{\half} (1+q)^{-1} \y{1} \left(1- \y{2} \right) \, ,
\qquad \u{2}{1} \u{2}{3} =\y{1} \y{3} 
\\
& \u{2}{1} \u{3}{3} = -q^{\half} (1+q)^{-1} (1 + q ~ \y{2}) \y{3} 
\end{align*}
\begin{align*}
& \u{3}{1}\u{1}{3} = (1+q)^{-1}(1-\y{2} - q^{-\half} \y{1}\y{3})
\\
&\u{3}{1} \u{2}{3} = q^\frac{1}{2} (1+q)^{-1} ~ \left(1 - \y{2} \right) \y{3} \, ,
\qquad \u{3}{1}   \u{3}{3} = -q^{\frac{3}{2}} (1+q)^{-1} \y{3}^2 \, .
\end{align*}

\bigskip\bigskip
\noindent
\textbf{Acknowledgments.}~\\[.5em] 
GL is grateful to Irene Sabadini for the nice invitation at the 
 ICCA13 Conference in Holon, Israel in June 2022 and to Elena Luna for the great hospitality in Holon. He also thanks all other organisers and participants for the very pleasant time at the conference. \\
GL acknowledges partial support from INFN, Iniziativa Specifica GAST
and from INdAM-GNSAGA.
GL acknowledges support from PNRR MUR projects PE0000023-NQSTI.\\
CP is grateful to the Department of Mathematics, Informatics and Geosciences of Trieste University for the hospitality.  CP was partially supported by COST Actions CaLISTA CA 21109 and CaLIGOLA MSCA-2021-SE-01-101086123, and from INdAM-GNSAGA.

\end{document}